\documentclass[a4paper]{article}
\usepackage{amsthm,amssymb,amsmath,enumerate}
\usepackage{wrapfig,subcaption,float,graphicx,pict2e}
\usepackage{cite}

\usepackage{thmtools}
\usepackage{thm-restate}

\usepackage{tikz}
\usetikzlibrary{calc}
\usetikzlibrary{backgrounds}
\usetikzlibrary{decorations.pathreplacing}
\usetikzlibrary{decorations.pathmorphing}
\colorlet{hellgrau}{black!20!white}
\colorlet{dunkelgrau}{black!60!white}

\colorlet{grau}{black!40!white}
\colorlet{bold}{black} 
\tikzstyle{ledge}=[thick, grau]
\tikzstyle{rededge}=[very thick, bold]
\tikzstyle{lvertex}=[thick,circle,inner sep=0.cm, minimum size=2mm, fill=white, draw=grau]
\tikzstyle{redvx}=[thick,circle,inner sep=0.cm, minimum size=2mm, fill=white, draw=bold]

\tikzstyle{hvertex}=[thick,circle,inner sep=0.cm, minimum size=2mm, fill=white, draw=black]
\tikzstyle{hedge}=[very thick]
\tikzstyle{medge}=[thick]
\tikzstyle{harrow}=[thick,arrows=->]
\tikzstyle{darrow}=[thick,arrows=<-]
\tikzstyle{point}=[draw,circle,inner sep=0.cm, minimum size=1mm, fill=black]
\tikzstyle{pointer}=[thick,->,shorten >=2pt,color=dunkelgrau]
\tikzstyle{facebdry}=[color=auchblau, very thick] 
\tikzstyle{face}=[facebdry,fill=hellblau]
\tikzstyle{nface}=[color=hellblau,fill=hellblau,thick] 
\tikzset{>={latex}}
\tikzstyle{tinyvx}=[thick,circle,inner sep=0.cm, minimum size=1.3mm, fill=white, draw=black]
\tikzstyle{smallvx}=[hvertex,minimum size=1.7mm]

\newcommand{\comment}[1]{}
\newcommand{\N}{\mathbb N}
\newcommand{\R}{\mathbb R}

\newcommand{\EP}{Erd\H os-P\'osa}

\newcommand{\cP}{\mathcal{P}}

\newcommand{\EPP}{{E}rd{\H o}s-{P\'o}sa property}

\title{On the edge-Erd\H{o}s-P\'{o}sa property of Ladders}
\author{Raphael Steck\thanks{Ulm University, raphael-st@web.de} \hspace{0.01cm} and Arthur Ulmer\thanks{Ulm University, arthur.ulmer@gmx.net, supported by DFG, grant no.\  321904558}}
\date{\today}

\usepackage{hyperref}
\hypersetup{
pdftitle={On the edge-Erd\H{o}s-P\'{o}sa property of Ladders},
pdfauthor={Raphael Steck and Arthur Ulmer},
pdfsubject={Master Paper},
pdfproducer={pdfeTex 3.14159-1.30.6-2.2},
colorlinks=false,
pdfborder=0 0 0	
}

\usepackage{enumitem}

\theoremstyle{plain}
\newtheorem{theo}{Theorem}
\newtheorem{lemma}[theo]{Lemma}

\newtheorem{conjecture}[theo]{Conjecture}
\newtheorem{theorem}[theo]{Theorem}

\newtheorem{proposition}[theo]{Proposition}

\newtheorem*{theo*raphael}{Theorem \ref{no14rungs}}
\newtheorem*{theo*arthur}{Theorem \ref{3RungLadderHasEdgeEPProperty}}

\begin{document}

\maketitle




\begin{abstract}
We prove that the ladder with $3$~rungs and the house graph have the edge-Erd\H{o}s-P\'{o}sa property, while ladders with $14$~rungs or more have not. Additionally, we prove that the latter bound is optimal in the sense that the only known counterexample graph does not permit a better result.
\end{abstract}

\section{Introduction} \label{sec:intro}

For a graph $H$, can we find many (vertex-)disjoint subgraphs that contain $H$ as a minor in a graph $G$? If not, can we find a bounded (vertex) set $S$ in $G$, called the \emph{hitting set}, such that $G - S$ does not contain $H$ as a minor?
If at least one of those statements holds for every $k\in\mathbb{N}$ and every graph $G$, then we say that $H$ has the \emph{Erd\H{o}s-P\'{o}sa property}. 
When we replace vertices with edges and search for edge-disjoint graphs and an edge hitting set, we arrive at the \emph{edge-Erd\H{o}s-P\'{o}sa property}.

This field owes its name to a result for cycles by Erd\H{o}s and P\'{o}sa in 1965 \cite{erdosposa65}.
Since then, the vertex version has been well researched. Robertson and Seymour \cite{robertson86} have shown as early as 1986 that $H$ has the (vertex-)Erd\H{o}s-P\'{o}sa property if and only if $H$ is a planar graph. 
The edge-Erd\H{o}s-P\'{o}sa property, however,  behaves quite differently. 
It is still true that non-planar graphs do not have the edge-Erd\H{o}s-P\'{o}sa property, as stated by Raymond and Thilikos in 2017 \cite{raymond17}, but the edge version fails even for some very simple planar graphs:
Bruhn, Heinlein and Joos \cite{bruhn18} have shown in 2018 that subcubic trees with large enough pathwidth and also ladders (see Figure~\ref{fig:LadderWith5Rungs}) with $71$~rungs or more do not have the edge-Erd\H{o}s-P\'{o}sa property.

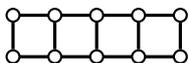
\begin{figure}[hbt] 
	\centering
	\begin{tikzpicture}[scale=1.1]
	\tikzstyle{ded}=[line width=0.8pt,double distance=1.2pt,draw=white,double=black]
	\tikzstyle{bubble}=[color=hellgrau,line width=6pt,fill=hellgrau,rounded corners=4pt]
	\def\runglen{0.5}
	\def\ladderlen{5}

	\clip (-0.5*\runglen*\ladderlen,-\runglen)
	rectangle (0.5*\runglen*\ladderlen,\runglen);

	\begin{scope}[shift={(0.5*\ladderlen*\runglen-0.5*\runglen,0)}]
	\draw[hedge] (0,0.5*\runglen) -- (-\ladderlen*\runglen+\runglen,0.5*\runglen);
	\draw[hedge] (0,-0.5*\runglen) -- (-\ladderlen*\runglen+\runglen,-0.5*\runglen);
	\foreach \i in {1,...,\ladderlen}{
	  \node[smallvx] (C\i) at (-\i*\runglen+\runglen,0.5*\runglen){};
	  \node[smallvx] (D\i) at (-\i*\runglen+\runglen,-0.5*\runglen){};
	  \draw[hedge] (C\i) -- (D\i);
	}
	\end{scope}	
	\end{tikzpicture}
\caption{A ladder with 5 rungs}\label{fig:LadderWith5Rungs}
\end{figure}

For positive results, as far as we know, the only graphs that are known to have the edge-Erd\H{o}s-P\'{o}sa property are cycles \cite{BHJ19,BJU19} (in particular $K_3$ \cite{DR05}), $K_4$ \cite{BH18} and $\Theta_r$ \cite{raymond17} (the multigraph on two vertices with $r$ parallel edges between them). The results on cycles and $\Theta_r$ are not that hard to come by. However, these graphs are also one of the simplest types of graphs there are. In contrast to that, the proof for $K_4$ is a lot of work even though a $K_4$ is just a slightly more complex graph than a $4$-cycle or a $\Theta_3$. 
As there are only few results, we speculate that positive results for the edge version are just that much harder to obtain. The complexity of the proof of the following result of ours supports this claim.

\begin{theorem} \label{3RungLadderHasEdgeEPProperty}
The ladder with $3$~rungs has the edge-{E}rd{\H o}s-{P\'o}sa property.
\end{theorem}

From the proof of Theorem \ref{3RungLadderHasEdgeEPProperty} we can deduce that the house graph (see Chapter \ref{SmallLadder}) has the edge-Erd\H{o}s-P\'{o}sa property as well. 
Again, both of these graphs are just $6$-cycles respectively $5$-cycles with a chord (and both of them are small expansions of $\Theta_3$), but the proof takes quite some effort.
As a byproduct, we see that $(A,m)$-trees, that is trees that contain $m$ vertices of a given vertex set $A$, have the edge-Erd\H{o}s-P\'{o}sa property, too. These can be seen as a generalization of $A$-paths.

To clarify what happens for larger ladders, we improve upon the negative result by Bruhn, Heinlein and Joos \cite{bruhn18} mentioned above.

\begin{restatable}{theorem}{theoremFourteenRungs}  \label{no14rungs}
Ladders with 14 rungs or more do not have the edge-Erd\H{o}s-P\'{o}sa property.
\end{restatable}

The key ingredient for the negative proofs is a gadget that Bruhn et al.\ call \emph{condensed wall} (see Figure~\ref{fig:condWall}). As of now, all planar graphs that have the edge-Erd\H{o}s-P\'{o}sa property are minors of the condensed wall and all that do not have it are not minors of it. Therefore, the condensed wall might play a key role in trying to characterize the graphs that have the edge-Erd\H{o}s-P\'{o}sa property.
This leads to the following conjecture by Bruhn et al. \cite{bruhn18}

\begin{conjecture}[Bruhn, Heinlein and Joos \cite{bruhn18}] \label{intro:conjectureCondWall}
Let $H$ be a planar graph such that there is an integer $r$ such that the condensed wall of size $r$ contains $H$ as a minor. Then $H$ has the edge-Erd\H{o}s-P\'{o}sa property.
\end{conjecture}

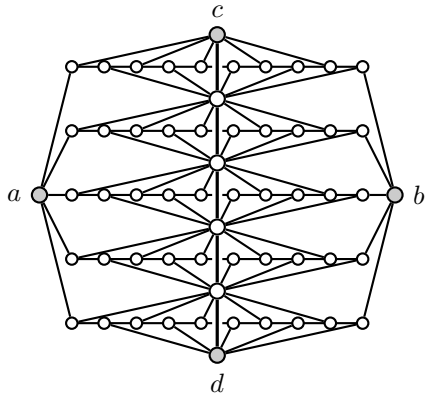
\begin{figure}[bht] 
\centering
\begin{tikzpicture}[scale=0.85]
\tikzstyle{tinyvx}=[thick,circle,inner sep=0.cm, minimum size=1.5mm, fill=white, draw=black]

\def\vstep{1}
\def\hstep{0.5}
\def\hwidth{9}
\def\hheight{4}

\def\totalheight{\hheight*\vstep}
\def\totalwidth{\hwidth*\hstep}
\pgfmathtruncatemacro{\minustwo}{\hwidth-2}
\pgfmathtruncatemacro{\minusone}{\hwidth-1}

\foreach \j in {0,...,\hheight} {
\draw[medge] (0,\j*\vstep) -- (\hwidth*\hstep,\j*\vstep);
\foreach \i in {0,...,\hwidth} {
\node[tinyvx] (v\i\j) at (\i*\hstep,\j*\vstep){};
}
}

\foreach \j in {1,...,\hheight}{
\node[hvertex] (z\j) at (0.5*\hwidth*\hstep,\j*\vstep-0.5*\vstep) {};
}
\pgfmathtruncatemacro{\plusvone}{\hheight+1}

\node[hvertex,fill=hellgrau,label=above:$c$] (z\plusvone) at (0.5*\totalwidth,\totalheight+0.5*\vstep) {};
\node[hvertex,fill=hellgrau,label=below:$d$] (z0) at (0.5*\totalwidth,-0.5*\vstep) {};

\foreach \j in {1,...,\plusvone}{
\pgfmathtruncatemacro{\subone}{\j-1}
\draw[line width=1.3pt,double distance=1.2pt,draw=white,double=black] (z\j) to (z\subone);
\foreach \i in {0,2,...,\hwidth}{
\draw[medge] (z\j) to (v\i\subone);
}
}

\foreach \j in {0,...,\hheight}{
\foreach \i in {1,3,...,\hwidth}{
\draw[medge] (z\j) to (v\i\j);
}
}

\pgfmathtruncatemacro{\minusvone}{\hheight-1}
\node[hvertex,fill=hellgrau,label=left:$a$] (a) at (-\hstep,0.5*\totalheight) {};
\foreach \j in {0,...,\hheight} {
\draw[medge] (a) -- (v0\j);
}

\node[hvertex,fill=hellgrau,label=right:$b$] (b) at (\totalwidth+\hstep,0.5*\totalheight) {};
\foreach \j in {0,...,\hheight} {
\draw[medge] (v\hwidth\j) to (b);
}
\end{tikzpicture}
\caption{A condensed wall of size 5.}
\label{fig:condWall}
\end{figure}

Theorem~\ref{no14rungs} is optimal in the sense that every counterexample based on a condensed wall (and recall, we do not know others) will not work for ladders with fewer than $14$~rungs: Indeed such ladders are minors of the condensed wall, see Proposition~\ref{tools:prop:13rungLadders}.

Do ladders with $4$ to $13$~rungs have the edge-Erd\H{o}s-P\'{o}sa property? 
We do not know, but this question can be seen as a test for the condensed wall gadget. If it is proven that ladders with 13~rungs still do not possess the edge-Erd\H{o}s-P\'{o}sa property, then clearly there is a counterexample graph not based on a condensed wall. However, if it is shown that ladders with 13~rungs have the edge-Erd\H{o}s-P\'{o}sa property, then this is yet another strong indication that the condensed wall plays a key role for this property.

For an overview and a comparison of the ordinary Erd\H{o}s-P\'{o}sa property and its edge counterpart, we recommend the appendix of \cite{bruhn17} and \cite{RST16}. 
A large list of Erd\H{o}s-P\'{o}sa results can be found on Raymond's webpage \cite{raymondweb}. 

The structure of this article is as follows: In Section~\ref{sec:prelim}, we define the objects we work with, such as the edge-\EPP{}, ladders and the condensed wall. Section~\ref{sec:toolbox} then contains the construction and proof that shows that long ladders do not have the edge-\EPP, that is, the proof of Theorem~\ref{no14rungs}. We will then move on to the small ladder and prove Theorem~\ref{3RungLadderHasEdgeEPProperty} in Section~\ref{SmallLadder}. In the same section, we also show that the House Graph has the edge-\EPP . Finally, we close by discussing our results in Section~\ref{sec:discussion}.

\section{Definitions} \label{sec:prelim}

In this section, we will state the basic definitions needed for this paper. 
Throughout, we will use standard notation as introduced in Diestel's textbook \cite{diestel2006}.

A class of graphs $\mathcal{H}$ has the \emph{edge-Erd\H{o}s-P\'{o}sa property} if there exists a function $f: \N \rightarrow \mathbb{R}$ such that for every graph $G$ and every integer $k$, there are $k$ edge-disjoint graphs in $G$ each isomorphic to a member of $\mathcal{H}$ or there is an edge set $X$ of $G$ of size at most $f(k)$ meeting all subgraphs in $G$ that are isomorphic to members of $\mathcal{H}$. We call $f$ an \emph{edge-\EP\ function} for  $\mathcal{H}$. To simplify notation, we say that a graph $H$ has the edge-\EPP\ if the set of all graphs that contain a minor isomorphic to $H$ has the edge-\EPP. Note that for a graph $H$ with a maximum degree~of~3, a graph $G$ contains $H$ as a minor if and only if it contains a subdivision of $H$ \cite{bruhn18}.

Very shortly we also need to deal with the (vertex-)\EPP: just replace all occurences of `edge/edges' by `vertex/vertices'

Next, we define an \emph{elementary ladder $L$} as a simple graph consisting of a vertex set $V(L) = \{u_1, u_2, \ldots , u_n, v_1, v_2, \ldots , v_n\}$ and edge set $E(L) = \{u_{i}u_{i+1} | i \in [n - 1] \} \cup \{v_{i}v_{i+1} | i \in [n - 1] \} \cup \{u_{i}v_{i} | i \in [n] \}$ for some $n \in \N, n \geq 1$.
The $n$ paths (of length one) $u_{i} v_{i}, i \in [n]$ are called \emph{rungs} of $L$, where the two paths (of length $n-1$) $u_1 u_2 u_3 \ldots u_n$ and $v_1 v_2 v_3 \ldots v_n$ are called \emph{stringers} of $L$.
The \emph{size} or \emph{length} of a ladder is given by its number of rungs, i.e. the size of $L$ is $n$.
A \emph{ladder} is a subdivision of an elementary ladder. We adapt the notion of rungs, stringers and size for ladders from their counterparts in elementary ladders.
A \emph{subladder} $L'$ a subgraph $L' \subseteq L$, where both $L'$ and $L$ are ladders and every rung $R$ of $L'$ is also an entire rung of $L$.

The main gadget used for proving that long ladders do not have the edge-Erd\H{o}s-P\'{o}sa property is a wall-like structure called a \emph{condensed wall} introduced by Bruhn et. al. \cite{bruhn18}. A condensed wall $W$ of size $r \in \N$ is the graph consisting of the following:
\begin{itemize}
\item For every $j \in [r]$, let $P^j = {u^j}_1, \ldots , {u^j}_{2r}$ be a path of length $2r - 1$ and for $j \in \{0\} \cup [r]$, let $z_j$ be a vertex. Moreover, let $a$, $b$ be two further vertices.
\item For every $i, j \in [r]$, add the edges $z_{j-1} {u^j}_{2i - 1}, z_{j} {u^j}_{2i}, z_{i-1} {z_i}, a {u^j}_{1}$ and $b {u^j}_{2r}$.
\end{itemize}
We define $c = z_0$ and $d = z_r$ and refer to
\begin{displaymath}
W_j = W[ \{ {u^j}_1, \ldots , {u^j}_{2r}, z_{j - 1}, z_j \} ]
\end{displaymath}

as the \emph{j-th layer of W}. 
Note that the layers of $W$ are precisely the blocks of $W - \{a, b\}$.
We will refer to the vertices $a, b$ quite often, and whenever we write $a$ or $b$ in this article, we refer to those vertices in a condensed wall.
The vertices connecting the layers of $W$ are $z_i, i \in \{0\} \cup [r]$, and we will call those \emph{bottleneck vertices}. This includes the vertices $c$ and $d$.
The edges $z_{i-1} z_i, i \in [r]$ are called \emph{jump-edges}.

For vertices $a, b, c, d$, an \emph{($a$-$b$, $c$-$d$)-linkage} is the vertex-disjoint union of an $a$-$b$-path with a $c$-$d$-path.

\section{Long Ladders} \label{sec:toolbox}

To show that a ladder $L$ does not have the edge-Erd\H{o}s-P\'{o}sa property, we explicitly construct a counterexample graph~$G^*$ for every given length $l$ of the ladder and size $r-1$ of the hitting set. So let $r \geq 2$ and $l \geq 14$ be some integers. Then we construct $G^*$ in the following way:

\begin{itemize}
\item Let $l_A \geq 7$ and $l_C \geq 4$ such that $l_A + l_C + 3 = l$.
\item Let $L_A$ be an elementary ladder of length $l_A$ and let $L_C$ be an elementary ladder of length $l_C$.
\item Let $G^*$ be the union of $L_A$ and $L_C$, where every edge is replaced by $r$~internally disjoint paths of length~$2$.
\item Add a condensed wall $W$ of size $2r$ to $G^*$.
\item Connect the stringers at one end of $L_A$ and $L_C$ to $W$ with $r$ internally disjoint paths of length~$2$ each as in Figure~\ref{fig:modifiedGadget}.
\end{itemize}

A depiction of $G^*$ can be found in Figure~\ref{fig:modifiedGadget}.
Let $A$ be the component of $G^* - W$ that contains a ladder with $l_A$~rungs, and let $C$ be the other component of $G^* - W$.

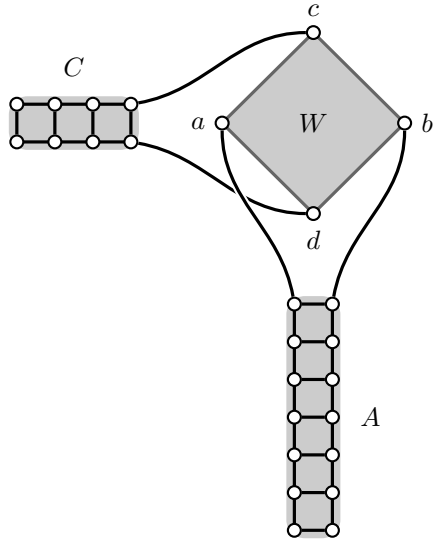
\begin{figure}[bht] 
\centering
\begin{tikzpicture}[scale=1.0]
\tikzstyle{ded}=[line width=0.8pt,double distance=1.2pt,draw=white,double=black]
\tikzstyle{bubble}=[color=hellgrau,line width=6pt,fill=hellgrau,rounded corners=4pt]
\def\step{1.2}
\def\runglen{0.5}
\def\ladderlen{4}
\def\offset{2*\step}

\clip (-\runglen*\ladderlen-2.2,-\runglen*\ladderlen-4.0)
rectangle (2,2);


\node[smallvx,label=left:$a$] (a) at (-\step,0){};
\node[smallvx,label=right:$b$] (b) at (\step,0){};
\node[smallvx,label=above:$c$] (c) at (0,\step){};
\node[smallvx,label=below:$d$] (d) at (0,-\step){};

\begin{scope}[on background layer]
\draw[hedge,color=dunkelgrau,fill=hellgrau] (a.center) -- (c.center) -- (b.center) -- (d.center) -- cycle;
\end{scope}

\node at (0,0) {$W$};

\begin{scope}[shift={(-\offset,0)}]
\draw[hedge] (0,0.5*\runglen) -- (-\ladderlen*\runglen+\runglen,0.5*\runglen);
\draw[hedge] (0,-0.5*\runglen) -- (-\ladderlen*\runglen+\runglen,-0.5*\runglen);
\foreach \i in {1,...,\ladderlen}{
  \node[smallvx] (C\i) at (-\i*\runglen+\runglen,0.5*\runglen){};
  \node[smallvx] (D\i) at (-\i*\runglen+\runglen,-0.5*\runglen){};
  \draw[hedge] (C\i) -- (D\i);
}
\node at (-0.5*\ladderlen*\runglen+0.5*\runglen,0.5*\runglen+0.5) {$C$};
\end{scope}

\pgfmathtruncatemacro{\llen}{\ladderlen+3}

\begin{scope}[shift={(0,-\offset)}]
\draw[hedge] (-0.5*\runglen,0) -- (-0.5*\runglen,-\llen*\runglen+\runglen);
\draw[hedge] (0.5*\runglen,0) -- (0.5*\runglen,-\llen*\runglen+\runglen);
\foreach \i in {1,...,\llen}{
  \node[smallvx] (A\i) at (-0.5*\runglen,-\i*\runglen+\runglen){};
  \node[smallvx] (B\i) at (0.5*\runglen,-\i*\runglen+\runglen){};
  \draw[hedge] (A\i) -- (B\i);
}
\node at (0.5*\runglen+0.5,-0.5*\llen*\runglen+0.5*\runglen) {$A$};
\end{scope}

%
%
%
%

\begin{scope}[on background layer]
\draw[bubble] (A1.center) -- (A\llen.center) -- (B\llen.center) -- (B1.center) -- cycle; 
\draw[bubble] (C1.center) -- (C\ladderlen.center) -- (D\ladderlen.center) -- (D1.center) -- cycle; 
\end{scope}






\draw[ded,out=10,in=180] (C1) to (c);
\draw[ded,out=-10,in=180] (D1) to (d);

\draw[ded,out=100,in=-90] (A1) to (a);
\draw[ded,out=80,in=-90] (B1) to (b);
\end{tikzpicture}
\caption{Construction in Theorem~\ref{no14rungs} for a ladder with 14~rungs. Note that all edges drawn represent $r$ internally disjoint paths of length~$2$.}
\label{fig:modifiedGadget}
\end{figure}

To show that a ladder $L$ does not have the edge-Erd\H{o}s-P\'{o}sa property, we first show that we can always find a ladder of length $l$ in our counterexample graph $G^*$, even with $r - 1$ edges being deleted. This implies there can be no edge hitting set regardless of its size, as $r$ was arbitrary.
Afterwards, we will show that there can never be two edge-disjoint ladders of length $l$ in $G^*$, implying that the edge-Erd\H{o}s-P\'{o}sa property fails even for $k = 2$.
The first part will be relatively easy.

\subsection{Finding one ladder}

For our embedding in $W$, we introduce a graph called an \emph{X-wing}. An X-wing is a subgraph $G = L_1 \cup P_1 \cup \ldots \cup P_4$, where $L_1$ is a ladder with $3$~rungs whose first and last rung we call $R_1$ and $R_3$, respectively, and $P_1, P_2, P_3, P_4$ are pairwise (vertex-)disjoint paths such that $P_1$ is a $R_1$--$a$~path, $P_2$ is a $R_1$--$b$~path, $P_3$ is a $R_3$--$c$~path and $P_4$ is a $R_3$--$d$~path. Furthermore, we require the interior of $P_1, P_2, P_3$ and $P_4$ to be (vertex-)disjoint from $L_1$. For an illustration, see Figure~\ref{fig:LinkageWith3RungsInWall}.

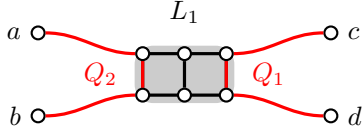
\begin{figure}[hbt] 
	\centering
	\begin{tikzpicture}[scale=1.1]
	\tikzstyle{ded}=[line width=0.8pt,double distance=1.2pt,draw=white,double=black]
	\tikzstyle{bubble}=[color=hellgrau,line width=6pt,fill=hellgrau,rounded corners=4pt]
	\def\runglen{0.5}
	\def\ladderlen{3}
	\def\offset{5*\runglen}
	
	\clip (-2*\runglen*\ladderlen-3.2,-1)
	rectangle (2*\runglen*\ladderlen+3.2,1.5);

	\begin{scope}[shift={(-0.5*\ladderlen*\runglen + 1.5*\runglen,0)}]
	\draw[hedge] (0,0.5*\runglen) -- (-\ladderlen*\runglen+\runglen,0.5*\runglen);
	\draw[hedge] (0,-0.5*\runglen) -- (-\ladderlen*\runglen+\runglen,-0.5*\runglen);
	\foreach \i in {1,...,\ladderlen}{
	  \node[smallvx] (C\i) at (-\i*\runglen+\runglen,0.5*\runglen){};
	  \node[smallvx] (D\i) at (-\i*\runglen+\runglen,-0.5*\runglen){};
	  \draw[hedge] (C\i) -- (D\i);
	}
	\node at (-0.5*\ladderlen*\runglen+0.5*\runglen,0.5*\runglen+0.5) {$L_1$};
	\node[smallvx,label=left:$a$] (a) at (-\ladderlen*\runglen - 1.5*\runglen,\runglen){};
	\node[smallvx,label=left:$b$] (b) at (-\ladderlen*\runglen - 1.5*\runglen,-\runglen){};
	\node[smallvx,label=right:$c$] (c) at (2.5*\runglen,\runglen){};
	\node[smallvx,label=right:$d$] (d) at (2.5*\runglen,-\runglen){};
	\end{scope}

		\begin{scope}[on background layer]
	\draw[bubble] (C1.center) -- (C\ladderlen.center) -- (D\ladderlen.center) -- (D1.center) -- cycle; 
	\end{scope}

	\draw[red, hedge] (C1)  -- (D1)  node [midway, xshift = 0.2cm, right] {$Q_1$} ;  
	\draw[red, hedge, out=0,in=180] (a) to (C\ladderlen);
	\draw[red, hedge, out=0,in=180] (b) to (D\ladderlen);
	
	\draw[red, hedge] (C\ladderlen)  -- (D\ladderlen)  node [midway, xshift = -0.2cm, left] {$Q_2$} ;  
	\draw[red, hedge, out=180,in=0] (c) to (C1);
	\draw[red, hedge, out=180,in=0] (d) to (D1);
	
	\end{tikzpicture}
\caption{An ($a$-$b$, $c$-$d$)-linkage in an X-wing}
\label{fig:LinkageWith3RungsInWall}
\end{figure}

\begin{lemma} \label{opt:lemma:3rungs}
In every condensed wall $W$ of size $2r$, there exists an X-wing, even when $r-1$ edges of $W$ are deleted.
\end{lemma}

\begin{proof}
After the deletion of $r-1$ edges, there are still two complete adjacent layers $W_{i-1}$ and $W_i$ of $W$ whose edges to $a$ and $b$ are still present. There, we can embed a graph $X$ as  in Figure~\ref{fig:3RungsInWall}.
In $W$, there were $2r$ edge-disjoint paths connecting $z_{i-2}$ to $c$ and another $2r$ edge-disjoint paths, disjoint from the first set, connecting $z_i$ to $d$ without using $a$ or $b$. So with $r-1$ edges being deleted, there is still at least one (in fact $r+1$) of each kind left. Together with $X$, they form an X-wing.
\qed
\end{proof}

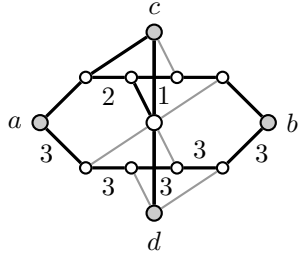
\begin{figure}[hbt] 
\centering
\begin{tikzpicture}[scale=1.2]
\tikzstyle{tinyvx}=[thick,circle,inner sep=0.cm, minimum size=1.5mm, fill=white, draw=black]

\def\vstep{1}
\def\hstep{0.5}
\def\hwidth{3}
\def\hheight{1}

\def\totalheight{\hheight*\vstep}
\def\totalwidth{\hwidth*\hstep}
\pgfmathtruncatemacro{\minustwo}{\hwidth-2}
\pgfmathtruncatemacro{\minusone}{\hwidth-1}

\foreach \j in {0,...,\hheight} {
\draw[ledge] (0,\j*\vstep) -- (\hwidth*\hstep,\j*\vstep);
\foreach \i in {0,...,\hwidth} {
\node[tinyvx] (v\i\j) at (\i*\hstep,\j*\vstep){};
}
}

\foreach \j in {1,...,\hheight}{
\node[hvertex] (z\j) at (0.5*\hwidth*\hstep,\j*\vstep-0.5*\vstep) {};
}
\pgfmathtruncatemacro{\plusvone}{\hheight+1}

\node[hvertex,fill=hellgrau,label=above:$c$] (z\plusvone) at (0.5*\totalwidth,\totalheight+0.5*\vstep) {};
\node[hvertex,fill=hellgrau,label=below:$d$] (z0) at (0.5*\totalwidth,-0.5*\vstep) {};

\foreach \j in {1,...,\plusvone}{
\pgfmathtruncatemacro{\subone}{\j-1}
\draw[line width=1.3pt,double distance=1.2pt,draw=white,double=black] (z\j) to (z\subone);
\foreach \i in {0,2,...,\hwidth}{
\draw[ledge] (z\j) to (v\i\subone);
}
}

\foreach \j in {0,...,\hheight}{
\foreach \i in {1,3,...,\hwidth}{
\draw[ledge] (z\j) to (v\i\j);
}
}

\pgfmathtruncatemacro{\minusvone}{\hheight-1}
\node[hvertex,fill=hellgrau,label=left:$a$] (a) at (-\hstep,0.5*\totalheight) {};
\foreach \j in {0,...,\hheight} {
\draw[ledge] (a) to (v0\j);
}

\node[hvertex,fill=hellgrau,label=right:$b$] (b) at (\totalwidth+\hstep,0.5*\totalheight) {};
\foreach \j in {0,...,\hheight} {
\draw[ledge] (v\hwidth\j) to (b);
}

	\draw[rededge] (z2)  -- (z1)  node [near end, xshift=-0.1cm, right] {$1$} ;  
	\draw[hedge] (z1)  to (v11) ;
	\draw[hedge] (z2)  to (v01) ;	
	\draw[rededge] (v01) -- (v11) node [midway, yshift=-0.25cm] {$2$}; 
	\draw[hedge] (v01) to (a)   ;	
	\draw[hedge] (v11) to (v21) ;
	\draw[hedge] (v21) to (v31) ;
	\draw[hedge] (v31) to (b)   ;
	\draw[rededge] (a)   -- (v00) node [midway, yshift=-0.1cm, left] {$3$}; 
	\draw[rededge] (v00) -- (v10) node [midway, yshift=-0.25cm] {$3$};
	\draw[rededge] (v10) -- (v20) node [very near end, yshift=-0.25cm] {$3$};
	\draw[rededge] (v20) -- (v30) node [midway, yshift=0.25cm] {$3$};
	\draw[rededge] (v30) -- (b)   node [midway, yshift=-0.1cm, right] {$3$};
	\draw[hedge] (z1)  to (z0)  ;
		

\end{tikzpicture}
\caption{An X-wing (thick edges) in a condensed wall of size 2. Edges belonging to rungs are labelled and red.}
\label{fig:3RungsInWall}
\end{figure}

The following lemma yields one part of the proof of Theorem~\ref{no14rungs}.

\begin{lemma} \label{tools:lem:oneDirectionOfMainProof}
For every integer $l \geq 14$ and every function $f: \N \rightarrow \R$, there exists an $r \in \N$ such that the graph $G^*$ (as defined in the beginning of the section) contains no set $B$ of at most $f(2)$ edges such that $G - B$ does not contain a ladder with $l$~rungs as a minor.
\end{lemma}

\begin{proof}
We choose $r$ such that $r-1 = f(2)$. Recall $G^*$ as defined in the beginning of Section~\ref{sec:toolbox}. Let $B \subseteq E(G)$ contain at most $f(2)$ edges.
By construction of $G^*$, there are $r$ edge-disjoint ladders of size $l_A$ in $A$. Thus in $G^* - B$, we can find a ladder $L_1$ with $l_A$~rungs in $A-B$. Similarly, there is a ladder $L_2$ with $l_C$ rungs in $C-B$. Lemma~\ref{opt:lemma:3rungs} yields an X-wing $X$ in $W-B$. As there are $r$~disjoint connections between each stringer of $L_1, L_2$ and $W$, we can connect $L_1, L_2$ and $X$ to form a ladder with $l_A + 3 + l_C = l$~rungs in $G^*-B$.
\qed
\end{proof}

\subsection{Excluding two ladders}

The tactic for the proof of Theorem~\ref{no14rungs} will be to show that every ladder $U$ in $G^*$ must contain an X-wing and thus an ($a$-$b$, $c$-$d$)-linkage in $W$, which can be only present once (see Lemma~\ref{def:lemma:noTwoLinkages}). However, to prove that $U$ contains an X-wing, we must exclude every other possibility of an embedding of $U$ in $G^*$. Therefore, we will need several technical lemmas to get an upper bound for how large a ladder in $W$ can be without containing an X-wing.

%


\begin{lemma} [Bruhn et al. \cite{bruhn18}] \label{def:lemma:noTwoLinkages}
A condensed wall $W$ (of any size) does not contain two edge-disjoint ($a$-$b$, $c$-$d$)-linkages.
\end{lemma}

When considering how large a ladder in a condensed wall can be, the two vertices $a$ and $b$ will be of great importance. Without them, we cannot form large ladders, as the next lemma will show:

\begin{lemma} \label{LadderInCondensedWall}
Let $W$ be a condensed wall. In $W - \{a, b\}$, the following holds:

\begin{enumerate}[label=\upshape(\roman*)]
\item Every ladder $L$ in $W - \{a, b\}$ has at most 5~rungs.  \emph{(Bruhn et al. \cite{bruhn18} )} 

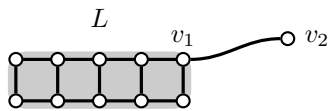
\begin{figure}[hbt] 
	\centering
	\begin{tikzpicture}[scale=1.1]
	\tikzstyle{ded}=[line width=0.8pt,double distance=1.2pt,draw=white,double=black]
	\tikzstyle{bubble}=[color=hellgrau,line width=6pt,fill=hellgrau,rounded corners=4pt]
	\def\runglen{0.5}
	\def\ladderlen{5}
	\def\offset{2.5*\runglen}
	
	\clip (-0.5*\runglen*\ladderlen,-1)
	rectangle (0.5*\runglen*\ladderlen+\offset+0.5,1.5);

	\begin{scope}[shift={(0.5*\ladderlen*\runglen,0)}]
	\draw[hedge] (0,0.5*\runglen) -- (-\ladderlen*\runglen+\runglen,0.5*\runglen);
	\draw[hedge] (0,-0.5*\runglen) -- (-\ladderlen*\runglen+\runglen,-0.5*\runglen);
	\foreach \i in {1,...,\ladderlen}{
	  \node[smallvx] (C\i) at (-\i*\runglen+\runglen,0.5*\runglen){};
	  \node[smallvx] (D\i) at (-\i*\runglen+\runglen,-0.5*\runglen){};
	  \draw[hedge] (C\i) -- (D\i);
	}
	\node at (-0.5*\ladderlen*\runglen+0.5*\runglen,0.5*\runglen+0.5) {$L$};
	\node at (0,\runglen) {$v_1$}; 
	\node[smallvx,label=right:$v_2$] (v2) at (\offset,\runglen){};

	\end{scope}

		\begin{scope}[on background layer]
	\draw[bubble] (C1.center) -- (C\ladderlen.center) -- (D\ladderlen.center) -- (D1.center) -- cycle; 
	\end{scope}


	\draw[hedge, out=180,in=0] (v2) to (C1);
	
	\end{tikzpicture}
\caption{A ladder with 5~rungs and a path}
\label{fig:LadderWith5RungsAndOnePath}
\end{figure}

\item There is no embedding of a graph as in Figure~\ref{fig:LadderWith5RungsAndOnePath} in $W - \{a, b\}$ such that the branch vertex corresponding to $v_1$ or $v_2$ is a bottleneck vertex.  

\item There is no embedding of a graph as in Figure~\ref{fig:LadderWith4RungsAndTwoPaths} or in Figure~\ref{fig:LadderWith3RungsAndTwoPathsAtSameEnd} in $W - \{a, b\}$ such that the branch vertices corresponding to one of $v_1$ and $v_2$ and one of $v_3$ and $v_4$ are bottleneck vertices.

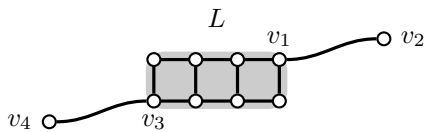
\begin{figure}[hbt] 
	\centering
	\begin{tikzpicture}[scale=1.1]
	\tikzstyle{ded}=[line width=0.8pt,double distance=1.2pt,draw=white,double=black]
	\tikzstyle{bubble}=[color=hellgrau,line width=6pt,fill=hellgrau,rounded corners=4pt]
	\def\runglen{0.5}
	\def\ladderlen{4}
	\def\offset{2.5*\runglen}
	
	\clip (-0.5*\runglen*\ladderlen-\offset-0.5,-1)
	rectangle (0.5*\runglen*\ladderlen+\offset+0.5,1.5);

	\begin{scope}[shift={(0.5*\ladderlen*\runglen,0)}]
	\draw[hedge] (0,0.5*\runglen) -- (-\ladderlen*\runglen+\runglen,0.5*\runglen);
	\draw[hedge] (0,-0.5*\runglen) -- (-\ladderlen*\runglen+\runglen,-0.5*\runglen);
	\foreach \i in {1,...,\ladderlen}{
	  \node[smallvx] (C\i) at (-\i*\runglen+\runglen,0.5*\runglen){};
	  \node[smallvx] (D\i) at (-\i*\runglen+\runglen,-0.5*\runglen){};
	  \draw[hedge] (C\i) -- (D\i);
	}
	\node at (-0.5*\ladderlen*\runglen+0.5*\runglen,0.5*\runglen+0.5) {$L$};
	\node at (0,\runglen) {$v_1$}; 
	\node[smallvx,label=right:$v_2$] (v2) at (\offset,\runglen){};
	\node at (-\ladderlen*\runglen + \runglen,-\runglen) {$v_3$}; 
	\node[smallvx,label=left:$v_4$] (v4) at (-\offset - \ladderlen*\runglen + \runglen,-\runglen){};

	\end{scope}

		\begin{scope}[on background layer]
	\draw[bubble] (C1.center) -- (C\ladderlen.center) -- (D\ladderlen.center) -- (D1.center) -- cycle; 
	\end{scope}


	\draw[hedge, out=180,in=0] (v2) to (C1);
	\draw[hedge, out=0,in=180] (v4) to (D\ladderlen);
	
	\end{tikzpicture}
\caption{A ladder with 4~rungs and two paths}
\label{fig:LadderWith4RungsAndTwoPaths}
\end{figure}

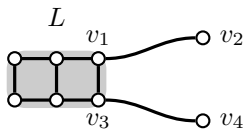
\begin{figure}[hbt] 
	\centering
	\begin{tikzpicture}[scale=1.1]
	\tikzstyle{ded}=[line width=0.8pt,double distance=1.2pt,draw=white,double=black]
	\tikzstyle{bubble}=[color=hellgrau,line width=6pt,fill=hellgrau,rounded corners=4pt]
	\def\runglen{0.5}
	\def\ladderlen{3}
	\def\offset{2.5*\runglen}
	
	\clip (-0.5*\runglen*\ladderlen,-1)
	rectangle (0.5*\runglen*\ladderlen+\offset+0.5,1.5);

	\begin{scope}[shift={(0.5*\ladderlen*\runglen,0)}]
	\draw[hedge] (0,0.5*\runglen) -- (-\ladderlen*\runglen+\runglen,0.5*\runglen);
	\draw[hedge] (0,-0.5*\runglen) -- (-\ladderlen*\runglen+\runglen,-0.5*\runglen);
	\foreach \i in {1,...,\ladderlen}{
	  \node[smallvx] (C\i) at (-\i*\runglen+\runglen,0.5*\runglen){};
	  \node[smallvx] (D\i) at (-\i*\runglen+\runglen,-0.5*\runglen){};
	  \draw[hedge] (C\i) -- (D\i);
	}
	\node at (-0.5*\ladderlen*\runglen+0.5*\runglen,0.5*\runglen+0.5) {$L$};
	\node at (0,\runglen) {$v_1$}; 
	\node[smallvx,label=right:$v_2$] (v2) at (\offset,\runglen){};
	\node at (0,-\runglen) {$v_3$}; 
	\node[smallvx,label=right:$v_4$] (v4) at (\offset,-\runglen){};

	\end{scope}

		\begin{scope}[on background layer]
	\draw[bubble] (C1.center) -- (C\ladderlen.center) -- (D\ladderlen.center) -- (D1.center) -- cycle; 
	\end{scope}


	\draw[hedge, out=180,in=0] (v2) to (C1);
	\draw[hedge, out=180,in=0] (v4) to (D1);
	
	\end{tikzpicture}
\caption{A ladder with 3~rungs and two paths at the same end}
\label{fig:LadderWith3RungsAndTwoPathsAtSameEnd}
\end{figure}

\end{enumerate}

\end{lemma}

\begin{proof}
(i) Every cycle in $W - \{a, b\}$ must contain either $z_i$ or $z_{i-1}$ (or both). When considering the smallest cycles in a ladder (those that contain two adjacent rungs), every vertex can only be part of two of them. Therefore, there can be at most four of those small cycles, which implies $L$ has size at~most~5. (The same result has also been proven in Lemma~9 in Bruhn et al. \cite{bruhn18}.)

(ii) + (iii) Let $W_i$ be the layer of $W$ that contains $L$. $W_i$ is connected to $W - W_i$ by only four vertices: $a, b, z_{i - 1}$ and $z_i$. Therefore, each ladder stringer that is continued outside of $W_i$ without using $a$ or $b$ must use one of $z_{i - 1}$ and $z_i$. 

Now when a ladder stringer of $L$ leaves $W_i$ (more precisely, when there is a path connected to the end of the ladder stringer that leaves $W_i$) through a vertex $z$, $z$ must necessarily lie on the end of $L$ (or not on $L$ at all). When $z_{i - 1}$ or $z_i$ lie at the end of $L$, this implies that $L$ is actually shorter~than~5, as every cycle in a layer~$W_i$ of $W$ must contain either $z_i$ or $z_{i-1}$ (or both).

Therefore, $L$ has length at~most~4 if $z_i$ or $z_{i - 1}$ lies on one of its ends. If both $z_i$ and $z_{i-1}$ lie on (different) ends of the ladder, $L$ is reduced to have a length of at~most~3. Finally, $L$ has at most 2~rungs if both $z_i$ and $z_{i-1}$ were on the same end.
\qed
\end{proof}

The following lemma will already yield a useful (and tight, see Proposition~\ref{tools:prop:13rungLadders}) upper bound on the size of a ladder $L$ in a condensed wall $W$.

\begin{lemma} \label{maxSizeInWall}
In every condensed wall $W$, every ladder $L$ contains at most 13~rungs.
\end{lemma}

\begin{proof}
Suppose there were a ladder $L$ in $W$ with $l \geq 14$~rungs.
Then $L - \{a, b\}$ contains at most three (disjoint) inclusion-maximal subladders. By construction, each of them contains neither $a$ nor $b$. Additionally, as $a$ and $b$ were part of at most $2$~rungs of $L$, those subladders must together contain at least $l - 2 \geq 12$~rungs.

First, suppose there were at most two such subladders $L_1, L_2$. Then by Lemma~\ref{LadderInCondensedWall}~(i), each may only contain up to $5$ rungs, which sums up to $5 + 5 = 10 < 12$ rungs, a contradiction.

So we can assume that there are exactly three (disjoint inclusion-maximal) subladders $L_1, L_2$ and $L_3$ of $L - \{a, b\}$.
Then two of them (say $L_1$ and $L_3$) must contain the ends of $L$, while one of them (say $L_2$) lies in between.
Via the stringers of $L$, $L_2$ has four disjoint paths connecting it to $L_1$ and $L_3$. By Lemma~\ref{LadderInCondensedWall}~(iii), $L_2$ can therefore only contain up to $3$~rungs. (Note that only two paths may contain $a$ or $b$.)

Similarly, $L_1$ and $L_3$ have each two (disjoint) paths connecting them to $L_2$. 
One of them must contain $a$ or $b$ for each of $L_1$ and $L_3$, which implies that the other cannot contain $a$ or $b$.
By Lemma~\ref{LadderInCondensedWall}~(ii), we conclude that $L_1$ and $L_3$ may therefore contain at most $4$~rungs. This sums up to $4 + 3 + 4 = 11 < 12$~rungs for the three of them, a contradiction.
\qed
\end{proof}

In the following, we will prove a series of lemmas that deal with ladders and connecting paths in a couple of special cases. They will be used for the proof of Theorem~\ref{no14rungs}.

\begin{lemma} \label{tools:lem:StringersOutAtSameEndViaAB}
Let $L$ be a ladder in a condensed wall $W$. Let $R$ be a rung at the end of $L$, with endvertices $v_1, v_2$. Let $P_1$ be a $v_1$-$a$-path and let $P_2$ be a $v_2$-$b$-path (both in $W$). Furthermore, let the interior of $P_1$ and $P_2$ be disjoint from each other and disjoint from $L$. (See Figure~\ref{fig:LadderWith6RungsAndABPaths})

\begin{figure}[hbt] 
	\centering
	\begin{tikzpicture}[scale=1.1]
	\tikzstyle{ded}=[line width=0.8pt,double distance=1.2pt,draw=white,double=black]
	\tikzstyle{bubble}=[color=hellgrau,line width=6pt,fill=hellgrau,rounded corners=4pt]
	\def\runglen{0.5}
	\def\ladderlen{6}
	\def\offset{2.5*\runglen}
	
	\clip (-0.5*\runglen*\ladderlen,-1)
	rectangle (0.5*\runglen*\ladderlen+\offset+0.5,1.5);

	\begin{scope}[shift={(0.5*\ladderlen*\runglen,0)}]
	\draw[hedge] (0,0.5*\runglen) -- (-\ladderlen*\runglen+\runglen,0.5*\runglen);
	\draw[hedge] (0,-0.5*\runglen) -- (-\ladderlen*\runglen+\runglen,-0.5*\runglen);
	\foreach \i in {1,...,\ladderlen}{
	  \node[smallvx] (C\i) at (-\i*\runglen+\runglen,0.5*\runglen){};
	  \node[smallvx] (D\i) at (-\i*\runglen+\runglen,-0.5*\runglen){};
	  \draw[hedge] (C\i) -- (D\i);
	}
	\node at (-0.5*\ladderlen*\runglen+0.5*\runglen,0.5*\runglen+0.5) {$L$};
	\node at (0,\runglen) {$v_1$}; 
	\node[smallvx,label=right:$a$] (a) at (\offset,\runglen){};
	\node at (0,-\runglen) {$v_2$}; 
	\node[smallvx,label=right:$b$] (b) at (\offset,-\runglen){};

	\end{scope}

		\begin{scope}[on background layer]
	\draw[bubble] (C1.center) -- (C\ladderlen.center) -- (D\ladderlen.center) -- (D1.center) -- cycle; 
	\end{scope}


	\draw[hedge, out=180,in=0] (a) to (C1);
	\draw[hedge, out=180,in=0] (b) to (D1);
	
	\end{tikzpicture}
\caption{A ladder with 6~rungs and paths to $a$ and $b$}
\label{fig:LadderWith6RungsAndABPaths}
\end{figure}
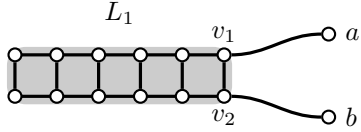
Let $L'$ be the union of $L, P_1$ and $P_2$.
Then $L'$ contains at most $6$ rungs.
\end{lemma}

Note that in this lemma and all following ones, we allow the paths $P_i$ to have a length of zero, i.e. to consist only of a single vertex (e.g. $v_1 = a$).

\begin{proof}
Let $L_1$ be the maximum subladder of $L - R$. Then $L_1$ cannot contain $a$ and $b$. By Lemma~\ref{LadderInCondensedWall}~(i), $L_1$ can therefore contain at most $5$~rungs. Together with $R$, this implies at most $6$~rungs for $L$.
\qed
\end{proof}

\begin{lemma} \label{tools:lem:StringersOutAtSameEndViaAandZ}
Let $L$ be a ladder in a condensed wall $W$. Let $R$ be a rung at the end of $L$, with endvertices $v_1, v_2$. Let $P_1$ be a $v_1$-$a$-path and let $P_2$ be a $v_2$-$z$-path (both in $W$), where $z$ is a bottleneck vertex. Furthermore, let the interior of $P_1$ and $P_2$ be disjoint from each other and disjoint from $L$. (See Figure~\ref{fig:LadderWith5RungsAndAZPaths})

Finally, let $L'$ be the union of $L, P_1$ and $P_2$ and let $L'$ not contain $b$.

\begin{figure}[hbt] 
	\centering
	\begin{tikzpicture}[scale=1.1]
	\tikzstyle{ded}=[line width=0.8pt,double distance=1.2pt,draw=white,double=black]
	\tikzstyle{bubble}=[color=hellgrau,line width=6pt,fill=hellgrau,rounded corners=4pt]
	\def\runglen{0.5}
	\def\ladderlen{5}
	\def\offset{2.5*\runglen}
	
	\clip (-0.5*\runglen*\ladderlen,-1)
	rectangle (0.5*\runglen*\ladderlen+\offset+0.5,1.5);

	\begin{scope}[shift={(0.5*\ladderlen*\runglen,0)}]
	\draw[hedge] (0,0.5*\runglen) -- (-\ladderlen*\runglen+\runglen,0.5*\runglen);
	\draw[hedge] (0,-0.5*\runglen) -- (-\ladderlen*\runglen+\runglen,-0.5*\runglen);
	\foreach \i in {1,...,\ladderlen}{
	  \node[smallvx] (C\i) at (-\i*\runglen+\runglen,0.5*\runglen){};
	  \node[smallvx] (D\i) at (-\i*\runglen+\runglen,-0.5*\runglen){};
	  \draw[hedge] (C\i) -- (D\i);
	}
	\node at (-0.5*\ladderlen*\runglen+0.5*\runglen,0.5*\runglen+0.5) {$L$};
	\node at (0,\runglen) {$v_1$}; 
	\node[smallvx,label=right:$a$] (a) at (\offset,\runglen){};
	\node at (0,-\runglen) {$v_2$}; 
	\node[smallvx,label=right:$z$] (z) at (\offset,-\runglen){};

	\end{scope}

	\begin{scope}[on background layer]
	\draw[bubble] (C1.center) -- (C\ladderlen.center) -- (D\ladderlen.center) -- (D1.center) -- cycle; 
	\end{scope}


	\draw[hedge, out=180,in=0] (a) to (C1);
	\draw[hedge, out=180,in=0] (z) to (D1);
	
	\end{tikzpicture}
\caption{A ladder with 5~rungs and paths to $a$ and a bottleneck vertex $z$}
\label{fig:LadderWith5RungsAndAZPaths}
\end{figure}
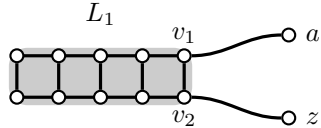

Then $L$ contains at most $5$ rungs.
\end{lemma}

\begin{proof}
Let $L_1$ be the maximum subladder of $L - R$. Then $L_1$ cannot contain $a$ and $b$. If $L_1$ contains at least $2$ rungs, it must be entirely contained in a single layer $W_i$ of $W$ by Lemma~\ref{LadderInCondensedWall}~(i). Furthermore, the stringers connecting $L_1$ to $R$ and the paths $P_1$ and $P_2$ form two disjoint paths continuing the stringers of $L_1$ to vertices outside of $W_i$. 
By Lemma~\ref{LadderInCondensedWall}~(ii), $L_1$ can therefore contain at most $4$~rungs. Together with $R$, this implies at most $5$~rungs for $L$.
\qed
\end{proof}

\begin{lemma} \label{tools:lem:twoLaddersWithAZPaths}
Let $L_1, L_2$ be disjoint ladders in a condensed wall $W$. Let $R_1$ be a rung at the end of $L_1$, with endvertices $v_1, v_2$. Similarly, let $R_2$ be a rung at the end of $L_2$, with endvertices $v_3, v_4$. Let $P_1$ be a $v_1$-$a$-path and let $P_2$ be a $v_4$-$z$-path (both in $W$), where $z$ is some vertex in a layer that is not occupied by $L_1$.
Additionally, let $P_3$ be a $v_2$-$v_3$-path in $W$. Furthermore, let the interior of $P_1, P_2$ and $P_3$, as well as $L_1$ and $L_2$, be pairwise disjoint. (See Figure~\ref{fig:TwoLaddersWithAZPaths})

\begin{figure}[hbt] 
\centering
	\begin{tikzpicture}[scale=1.0]
	\tikzstyle{ded}=[line width=0.8pt,double distance=1.2pt,draw=white,double=black]
	\tikzstyle{bubble}=[color=hellgrau,line width=6pt,fill=hellgrau,rounded corners=4pt]
	\def\runglen{0.5}
	\def\ladderlenone{9}
	\def\ladderlentwo{2}
	\def\offset{2.5*\runglen}
	
	\clip (-\runglen*\ladderlenone-\offset,-1)
	rectangle (\runglen*\ladderlentwo+\offset,1.5);

	\begin{scope}[shift={(-\offset,0)}]
	\draw[hedge] (0,0.5*\runglen) -- (-\ladderlenone*\runglen+\runglen,0.5*\runglen);
	\draw[hedge] (0,-0.5*\runglen) -- (-\ladderlenone*\runglen+\runglen,-0.5*\runglen);
	\foreach \i in {1,...,\ladderlenone}{
	  \node[smallvx] (C\i) at (-\i*\runglen+\runglen,0.5*\runglen){};
	  \node[smallvx] (D\i) at (-\i*\runglen+\runglen,-0.5*\runglen){};
	  \draw[hedge] (C\i) -- (D\i);
	}
	\node at (-0.5*\ladderlenone*\runglen+0.5*\runglen,0.5*\runglen+0.5) {$L_1$};
	\node at (0,\runglen) {$v_1$};
	\node at (0,-\runglen) {$v_2$};
	\node[smallvx,label=above:$a$] (a) at (\offset,1*\runglen){};
	\end{scope}
	
	\begin{scope}[shift={(\offset,0)}]
	\draw[hedge] (0,0.5*\runglen) -- (\ladderlentwo*\runglen-\runglen,0.5*\runglen);
	\draw[hedge] (0,-0.5*\runglen) -- (\ladderlentwo*\runglen-\runglen,-0.5*\runglen);
	\foreach \i in {1,...,\ladderlentwo}{
	  \node[smallvx] (A\i) at (\i*\runglen-\runglen,0.5*\runglen){};
	  \node[smallvx] (B\i) at (\i*\runglen-\runglen,-0.5*\runglen){};
	  \draw[hedge] (A\i) -- (B\i);
	}
	\node at (0.5*\ladderlentwo*\runglen-0.5*\runglen,0.5*\runglen+0.5) {$L_2$};
	\node at (0,\runglen) {$v_3$};
	\node at (0,-\runglen) {$v_4$};
	\node[smallvx,label=below:$z$] (z) at (-\offset,-1*\runglen){};

	\end{scope}
	

	\begin{scope}[on background layer]
	\draw[bubble] (A1.center) -- (A\ladderlentwo.center) -- (B\ladderlentwo.center) -- (B1.center) -- cycle; 
	\draw[bubble] (C1.center) -- (C\ladderlenone.center) -- (D\ladderlenone.center) -- (D1.center) -- cycle; 
	\end{scope}

	\draw[ded,out=0,in=180] (D1) to (A1);
	
	\draw[ded,out=180,in=0] (a) to (C1);
	\draw[ded,out=0,in=180] (z) to (B1);
	
	\end{tikzpicture}
\caption{Two ladders with paths to $a$ and a bottleneck vertex $z$}
\label{fig:TwoLaddersWithAZPaths}
\end{figure}
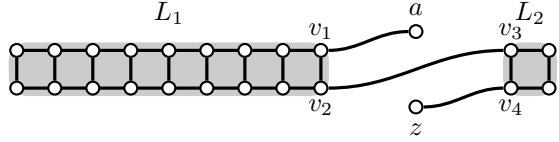

Then the sum of the number of rungs in $L_1$ and $L_2$ is at most $12$~rungs.
\end{lemma}

\begin{proof}
Let $L'_1$ be the maximum subladder of $L_1 - R_1$. Then both $L_2$  and $L'_1$ cannot contain $a$. 

Now suppose neither of them would contain $b$, either. Then by Lemma~\ref{LadderInCondensedWall}~(i), each of $L_2$ and $L'_1$ can contain at most $5$~rungs. Together with $R_1$, this means up to $11$~rungs for $L_1$ and $L_2$.

Now suppose $b$ would be contained in $L_2$. If $b$ is contained in $R_2$, we define $L'_2$ as the (unique) maximal subladder of $L_2 - R_2$. Again, we apply Lemma~\ref{LadderInCondensedWall}~(i) to see that $L'_1$ and $L'_2$ can have at most $5$ rungs each. This implies at most $6$~rungs each for $L_1$ and $L_2$, which sums up to at most $12$~rungs in total.

If $b$ is contained in $L_2 - R_2$, there is a (maximal) proper subladder $L'_2$ of $L_2 - b$ that contains $R_2$ and has four disjoint paths continuing its stringers to vertices outside of its layer: $P_2$, $P_3$ and the parts of the stringers of $L_2$ that connect $L'_2$ with the rungs of $L_2 - L'_2$. However, as $a$ is contained in neither of them (as $a$ is on $P_1$), at least three of them contain neither $a$ nor $b$. This is a contradiction.

Finally, we can suppose $b$ is contained in $L'_1$. Then $L_2$ contains at most $2$ rungs due to Lemma~\ref{LadderInCondensedWall}~(iii).

Now $L'_1 - b$ contains at most two disjoint maximal subladders. If it would contain only one such maximal subladder $L_3$, then $L_3$ can contain at most $5$ rungs by Lemma~\ref{LadderInCondensedWall}~(i) as it contains neither $a$ nor $b$. This implies up to $7$ rungs for $L_1$, which sums up to at most $7 + 2 = 9$ rungs for $L_1$ and $L_2$.

So let us suppose $L'_1 - b$ contains exactly two disjoint maximal subladders $L_3$ and $L_4$, where $L_3$ is the subladder adjacent to $R_1$. Then $L_4$ can contain at most $4$~rungs by Lemma~\ref{LadderInCondensedWall}~(ii), and $L_3$~can contain at most $3$~rungs by Lemma~\ref{LadderInCondensedWall}~(iii). Together with at most $1$ rung that can be incident with $b$ and $R_1$, this implies up to $4 + 1 + 3 + 1 = 9$ rungs for $L_1$. With $2$ rungs in $L_2$, we get a sum of rungs of at most $11$.
\qed
\end{proof}

\begin{lemma} \label{tools:lem:StringersOutAtCD}

Let $L$ be a ladder in a condensed wall $W$. Let $R$ be a rung at the end of $L$, with endvertices $v_1, v_2$. Let $P_1$ be a $v_1$-$c$-path and let $P_2$ be a $v_2$-$d$-path (both in $W$). Furthermore, let the interior of $P_1$ and $P_2$ be disjoint from each other and disjoint from $L$.

Let $Q = (P_1 - v_1) + (P_2 - v_2)$. (See Figure~\ref{fig:LadderWithCDPaths})

\begin{figure}[hbt] 
	\centering
	\begin{tikzpicture}[scale=1.1]
	\tikzstyle{ded}=[line width=0.8pt,double distance=1.2pt,draw=white,double=black]
	\tikzstyle{bubble}=[color=hellgrau,line width=6pt,fill=hellgrau,rounded corners=4pt]
	\def\runglen{0.5}
	\def\ladderlen{9}
	\def\offset{2.5*\runglen}
	
	\clip (-0.5*\runglen*\ladderlen,-1)
	rectangle (0.5*\runglen*\ladderlen+\offset+0.5,1.5);

	\begin{scope}[shift={(0.5*\ladderlen*\runglen,0)}]
	\draw[hedge] (0,0.5*\runglen) -- (-\ladderlen*\runglen+\runglen,0.5*\runglen);
	\draw[hedge] (0,-0.5*\runglen) -- (-\ladderlen*\runglen+\runglen,-0.5*\runglen);
	\foreach \i in {1,...,\ladderlen}{
	  \node[smallvx] (C\i) at (-\i*\runglen+\runglen,0.5*\runglen){};
	  \node[smallvx] (D\i) at (-\i*\runglen+\runglen,-0.5*\runglen){};
	  \draw[hedge] (C\i) -- (D\i);
	}
	\node at (-0.5*\ladderlen*\runglen+0.5*\runglen,0.5*\runglen+0.5) {$L$};
	\node[bold] at (0.5,0) {$Q$};
	\node at (0,\runglen) {$v_1$}; 
	\node[redvx,label=right:$c$] (c) at (\offset,\runglen){};
	\node at (0,-\runglen) {$v_2$}; 
	\node[redvx,label=right:$d$] (d) at (\offset,-\runglen) {};

	\end{scope}

		\begin{scope}[on background layer]
	\draw[bubble] (C1.center) -- (C\ladderlen.center) -- (D\ladderlen.center) -- (D1.center) -- cycle; 
	\end{scope}


	\draw[rededge, out=180,in=0] (c) to (C1);
	\draw[rededge, out=180,in=0] (d) to (D1);

	\end{tikzpicture}
\caption{A ladder with paths to $c$ and $d$}
\label{fig:LadderWithCDPaths}
\end{figure}
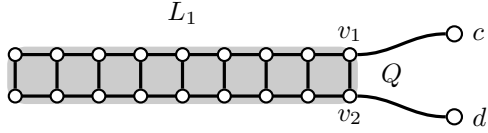

\begin{enumerate}[label=\upshape(\roman*)]
\item 
If both $L$ and $Q$ contain neither $a$ nor $b$, $L$ has at most 2~rungs.
\item 
If $L$ contains neither $a$ nor $b$, and there is exactly one of $a$ and $b$ on $Q$, $L$ has at most 4~rungs.
\item 
If $L$ contains neither $a$ nor $b$, and there are both $a$ and $b$ on $Q$, $L$ has at most 5~rungs. 
\item 
If $L$ contains exactly one of $a$ and $b$, but $Q$ contains neither of them, $L$ has at most 5~rungs.
\item 
If $L$ contains exactly one of $a$ and $b$, and there is exactly one of $a$ and $b$ on $Q$, too, $L$ has at most 8~rungs.
\item 
If $L$ contains both $a$ and $b$ (which implies $Q$ contains neither of them), $L$ has at most 9~rungs.
\end{enumerate}

\end{lemma}

\begin{proof}
(i) This follows from Lemma~\ref{LadderInCondensedWall}~(iii).

(ii) As $L$ contains neither $a$ nor $b$, if $L$ has more than one rung, it must be entirely contained in a single layer of $W$.
Additionally, as $Q$ is allowed to contain only one of $a$ and $b$, either $P_1$ or $P_2$ must not contain $a$ and $b$. This implies $L$ has a stringer continued to $c$ or $d$ without using $a$ and $b$.
By Lemma~\ref{LadderInCondensedWall}~(ii), $L$ can thus have at most 4~rungs.

(iii) As in (ii), $L$ must be contained in a single layer of $W$. By Lemma~\ref{LadderInCondensedWall}~(i), $L$ can therefore have at most 5~rungs. 

Before we continue with the proof of~(iv), we make some observations that will prove useful for the next three cases.
Let $P$ be the path in $L \cup Q$ formed by $P_1, P_2$ and~$R$.

\vspace{0.2cm}
\noindent \textbf{Claim~1:} \emph{If $P$ uses neither $a$ nor $b$, $P$ must use all the vertices $z_i, i \in \{0\} \cup [r]$, where $z_0 = c$ and $z_r = d$.}
\vspace{0.2cm}

By definition, $P$ uses $z_0 = c$ and $z_r = d$. Furthermore, $a, b$ and $z_i$ disconnect $c$ from $d$ for every $i \in \{1, \ldots, n-1\}$.

\vspace{0.2cm}
\noindent \textbf{Claim~2:} \emph{If $P$ contains neither $a$ nor $b$, and $L$ contains at most one of them, then $L$ has at most $2$~rungs.}
\vspace{0.2cm}

Without using a bottleneck vertex, all cycles in $W$ must contain both $a$ and~$b$.
As $L \cup Q$ contains at most one of $a$ and~$b$, if $L$ has more than one rung, all rungs of $L$ must be on a cycle with a bottleneck vertex. As $R$ is the only rung of $L$ that contains bottleneck vertices, all cycles of $L$ must contain $R$. As $R$ is required to be on the end of $L$, $L$ cannot have more than 2~rungs, no matter whether we use $a$ or $b$ or not.

\vspace{0.2cm}
\noindent \textbf{Claim~3:} \emph{If $P$ contains neither $a$ nor $b$, but $L$ contains both, then $L$ has at most $4$~rungs.}
\vspace{0.2cm}

Now there can be additional cycles that use no bottleneck vertex, but those must then contain both $a$ and $b$.
When considering the smallest cycles on $L$ (those that use two adjacent rungs), there can be only two cycles that contain both $a$ and $b$. Together with the single cycle that may be incident with a bottleneck vertex that we also got in the upper case, we are left with at most 3 smallest cycles in $L$. This mean we get an upper bound of 4~rungs for $L$.

Now, we can continue with the last three proofs:

(iv) If $P$ uses neither $a$ nor $b$ as in (i), $L$ has at most 2~rungs as seen in Claim~2. Therefore, we must only consider the case where $P$ uses $v \in \{a, b\}$.
As we know that $v$ is not a part of $Q$, we can conclude that $v$ lies on $r$.
Then we can apply Lemma~\ref{tools:lem:StringersOutAtSameEndViaAandZ} to see that $L$ contains at most $5$~rungs.

(v) Again, there is $v \in \{a, b\}$ contained in $L$.
If $v$ were contained in $R$, $L - R$ would contain neither $a$ nor $b$. Therefore, $L - R$ could only contain a subladder with at most 5~rungs due to Lemma~\ref{LadderInCondensedWall}~(i). This would imply at most $5 + 1 = 6$ rungs for $L$.

Now suppose $v$ is not in $R$. Then it is not a part of $P_1$ and $P_2$, either. As in (ii), since $Q$ is allowed to contain only one of $a$ and $b$, either $P_1$ or $P_2$ must not contain $a$ and $b$.
Let $L_1, L_2$ be the two maximal subladders of $L - v$, where $L_2$ is the one that is incident with $P_1$ and $P_2$. The stringers of $L$ form two disjoint paths connecting $L_1$ and $L_2$, where only one of them can contain $v$.
If $L_1$ or $L_2$ contains more that one rung, it is entirely contained in a single layer of $W$. We can therefore apply Lemma~\ref{LadderInCondensedWall}~(ii) and~(iii) to see that $L_1$ can only contain up to $4$~rungs, while there are at most $3$~rungs in $L_2$. Together with at most one rung that is incident with $v$, we get a total of $4 + 1 + 3 =8$~rungs as un upper blound for $L$.

(vi) If $P$ uses neither $a$ nor $b$, $L$ has at most 4~rungs as seen at the beginning of the proof. Therefore, we can conclude $P$ uses at least one of $a$ and $b$.
Thus, $L - R$ contains exactly one maximal subladder $L_1$ with only rung less that $L$, but only one of $a$ or $b$ can be contained in $L_1$.
As in the proof of~(v), we can argue that $L_1$ can only contain at most $8$~rungs as it has two stringers continued via disjoint paths to $c$ and $d$, but only one of the paths may contain $a$ or $b$.
This implies there are at most $8 + 1 = 9$~rungs in $L$.
\qed
\end{proof}

%
Lemma~\ref{tools:lem:StringersOutAtCD} can also be used when there are two ladders, as in the following lemma:

\begin{lemma} \label{tools:lem:AtMost13RungsWithCD}

Let $L_1, L_2$ be two disjoint ladders in a condensed wall $W$. Let $R_1$ be a rung at the end of $L_1$ with endvertices $v_1, v_2$. Similarly, let $R_2$ be a rung at the end of $L_2$ with endvertices $v_3, v_4$.

Now, let $P_1$ be a $v_1$-$c$-path in $W$ and let $P_2$ be a $v_4$-$d$-path in $W$. Additionally, let $P_3$ be a $v_2$-$v_3$-path in $W$. Finally, let the interior of $P_1, P_2$ and $P_3$, as well as $L_1$ and $L_2$, be pairwise disjoint. (See Figure~\ref{fig:TwoLaddersWithCDPaths})

\begin{figure}[hbt] 
\centering
	\begin{tikzpicture}[scale=1.0]
	\tikzstyle{ded}=[line width=0.8pt,double distance=1.2pt,draw=white,double=black]
	\tikzstyle{bubble}=[color=hellgrau,line width=6pt,fill=hellgrau,rounded corners=4pt]
	\def\runglen{0.5}
	\def\ladderlenone{8}
	\def\ladderlentwo{5}
	\def\offset{2.5*\runglen}
	
	\clip (-\runglen*\ladderlenone-1.5*\offset,-1)
	rectangle (\runglen*\ladderlentwo+1.5*\offset,1.5);

	\begin{scope}[shift={(-1.5*\offset,0)}]
	\draw[hedge] (0,0.5*\runglen) -- (-\ladderlenone*\runglen+\runglen,0.5*\runglen);
	\draw[hedge] (0,-0.5*\runglen) -- (-\ladderlenone*\runglen+\runglen,-0.5*\runglen);
	\foreach \i in {1,...,\ladderlenone}{
	  \node[smallvx] (C\i) at (-\i*\runglen+\runglen,0.5*\runglen){};
	  \node[smallvx] (D\i) at (-\i*\runglen+\runglen,-0.5*\runglen){};
	  \draw[hedge] (C\i) -- (D\i);
	}
	\node at (-0.5*\ladderlenone*\runglen+0.5*\runglen,0.5*\runglen+0.5) {$L_1$};
	\node at (0,\runglen) {$v_1$};
	\node at (0,-\runglen) {$v_2$};
	\node[smallvx,label=above:$c$] (c) at (\offset,1*\runglen){};
	\end{scope}
	
	\begin{scope}[shift={(1.5*\offset,0)}]
	\draw[hedge] (0,0.5*\runglen) -- (\ladderlentwo*\runglen-\runglen,0.5*\runglen);
	\draw[hedge] (0,-0.5*\runglen) -- (\ladderlentwo*\runglen-\runglen,-0.5*\runglen);
	\foreach \i in {1,...,\ladderlentwo}{
	  \node[smallvx] (A\i) at (\i*\runglen-\runglen,0.5*\runglen){};
	  \node[smallvx] (B\i) at (\i*\runglen-\runglen,-0.5*\runglen){};
	  \draw[hedge] (A\i) -- (B\i);
	}
	\node at (0.5*\ladderlentwo*\runglen-0.5*\runglen,0.5*\runglen+0.5) {$L_2$};
	\node at (0,\runglen) {$v_4$};
	\node at (0,-\runglen) {$v_3$};
	\node[smallvx,label=above:$d$] (d) at (-\offset,\runglen){};

	\end{scope}
	

	\begin{scope}[on background layer]
	\draw[bubble] (A1.center) -- (A\ladderlentwo.center) -- (B\ladderlentwo.center) -- (B1.center) -- cycle; 
	\draw[bubble] (C1.center) -- (C\ladderlenone.center) -- (D\ladderlenone.center) -- (D1.center) -- cycle; 
	\end{scope}

	\draw[ded,out=0,in=180] (D1) to (B1);
	
	\draw[ded,out=180,in=0] (c) to (C1);
	\draw[ded,out=0,in=180] (d) to (A1);
	
	\end{tikzpicture}
\caption{Two ladders with paths to $c$ and $d$}
\label{fig:TwoLaddersWithCDPaths}
\end{figure}
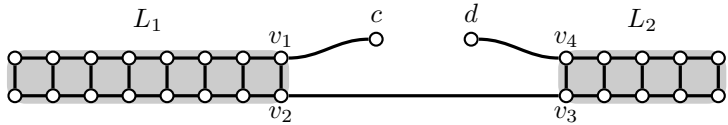

Then the sum of the number of rungs in $L_1$ and $L_2$ is at~most~13.

\end{lemma}

\begin{proof}

When the sum of rungs in $L_1$ and $L_2$ is more than~13, we can conclude that at least one of them (say $L_1$, as the situation is symmetric) must contain at least 7~rungs.
Then $L_1 - R_1$ must contain a subladder $L'_1$ with at least 6~rungs. Due to Lemma~\ref{LadderInCondensedWall}~(i), this is only possible if $L'_1$ contains $a$ or $b$. Say $L'_1$ contains $a$, as the situation is again symmetric.
Now $P_3, R_2$ and $P_2$ form together a path $P$ that continues a ladder stringer of $L_1$ to $d$. Additionally, $L_1$ has its other stringer continued at the same end via $P_1$ to $c$.
Let $Q_1 = (P_1 - v_1) \cup (P - v_2)$.

Now, we observe that $L_2$ is in a similar situation. $P_3, R_1$ and $P_1$ form a path $P'$ that continues a stringer of $L_2$ to $c$. Additionally, $L_2$ has its other stringer continued at the same end via $P_2$ to $d$.
Let $Q_2 = (P_2 - v_4) \cup (P' - v_3)$.
As $a$ is part of $L_1 - R_1$, we note that $Q_2$, $L_2$ and $Q_1$ cannot contain $a$.
Next, we take a look at where $b$ might be situated. We divide three cases:

First, $b$ might be a part of $L_1$. Then $L_1$ contains both $a$ and $b$. By Lemma~\ref{tools:lem:StringersOutAtCD}~(vi), $L_1$ can contain at most 9~rungs.
In this case, $b$ cannot be part of $L_2$ anymore. However, it could be a part of $R_1 \subset Q_2$. In this case, $L_2$ can contain up to 4~rungs by Lemma~\ref{tools:lem:StringersOutAtCD}~(ii). This sums up to $9 + 4 = 13$ rungs in total, which was what we wanted.
Alternatively, $b$ could not be part of $R_1$. Then $b$ is not a part of $Q_2$, which implies that $L_2$ can only contain two rungs by Lemma~\ref{tools:lem:StringersOutAtCD}~(i). This implies up to $9 + 2 = 11$ rungs in total.

Second, $b$ might be a part of $L_2$. Then $Q_2$ contains neither $a$ nor $b$. Using Lemma~\ref{tools:lem:StringersOutAtCD}~(iv), $L_2$ can have at most 5~rungs.
Similar to the first case, $b$ might lie on $R_2 \subset Q_1$, which implies that $L_1$ can contain at most 8~rungs due to Lemma~\ref{tools:lem:StringersOutAtCD}~(v). This means we get at most $8 + 5 = 13$ rungs in total.
Alternatively, $b$ might not lie on $R_2$, implying it is not on $Q_1$, either. Then we use Lemma~\ref{tools:lem:StringersOutAtCD}~(iv) to see that $L_1$ could only contain 5~rungs, which implies up to $ 5 + 5 = 10$ rungs in total.

Third and finally, $b$ might be a part of neither $L_1$ nor $L_2$. If it is on $P_1, P_2$ or $P_3$, it is part of both $Q_1$ and $Q_2$.
We can therefore conclude that $L_1$ contains at most 8~rungs by Lemma~\ref{tools:lem:StringersOutAtCD}~(v), while $L_2$ contains at most 4~rungs by Lemma~\ref{tools:lem:StringersOutAtCD}~(ii). This sums up to $8 + 4 = 12$ rungs in total.
Alternatively, $b$ could not be part of $P_1, P_2$ or $P_3$. Then, it is neither in $L_1, L_2, Q_1$ nor in $Q_2$. Therefore, we use Lemma~\ref{tools:lem:StringersOutAtCD}~(iv) to see that $L_1$ can contain only 5~rungs. Similarly, Lemma~\ref{tools:lem:StringersOutAtCD}~(i) implies that $L_2$ can only contain 2~rungs. In total, this means at most $5 + 2 = 7$ rungs in total.

As we have seen that an upper bound of 13~rungs in total hold true for all three cases, we have proven the lemma.
\qed
\end{proof}

\begin{lemma} \label{tools:lem:threeLaddersWithABCDPaths}
Let $L$ be a ladder where the pairs $\{a, b\}$ and $\{c, d\}$ are situated on the stringers of $L$ between different rungs, but both vertices of each pair lie between the same two rungs and on the same stringer of $L$.

Except for the part of the stringers that is between the two vertices of each pair, all of $L$ is situated in a condensed wall $W$. An example for the part of $L$ in $W$ is depicted in Figure~\ref{fig:ThreeLaddersWithABCDPaths}. 

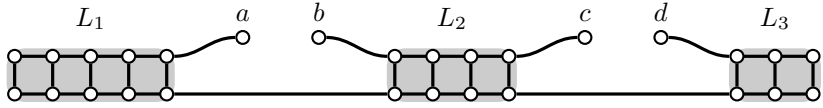
\begin{figure}[hbt] 
\centering
	\begin{tikzpicture}[scale=1.0]
	\tikzstyle{ded}=[line width=0.8pt,double distance=1.2pt,draw=white,double=black]
	\tikzstyle{bubble}=[color=hellgrau,line width=6pt,fill=hellgrau,rounded corners=4pt]
	\def\runglen{0.5}
	\def\ladderlenone{5}
	\def\ladderlentwo{4}
	\def\ladderlenthree{3}
	\def\offset{2*\runglen}
	
	\clip (-\runglen*\ladderlenone-1.5*\offset,-1)
	rectangle (\runglen*\ladderlentwo+\runglen*\ladderlenthree+4.5*\offset,1.5);

	\begin{scope}[shift={(-1.5*\offset,0)}]
	\draw[hedge] (0,0.5*\runglen) -- (-\ladderlenone*\runglen+\runglen,0.5*\runglen);
	\draw[hedge] (0,-0.5*\runglen) -- (-\ladderlenone*\runglen+\runglen,-0.5*\runglen);
	\foreach \i in {1,...,\ladderlenone}{
	  \node[smallvx] (C\i) at (-\i*\runglen+\runglen,0.5*\runglen){};
	  \node[smallvx] (D\i) at (-\i*\runglen+\runglen,-0.5*\runglen){};
	  \draw[hedge] (C\i) -- (D\i);
	}
	\node at (-0.5*\ladderlenone*\runglen+0.5*\runglen,0.5*\runglen+0.5) {$L_1$};
	\node[smallvx,label=above:$a$] (a) at (\offset,1*\runglen){};
	\end{scope}
	
	\begin{scope}[shift={(1.5*\offset,0)}]
	\draw[hedge] (0,0.5*\runglen) -- (\ladderlentwo*\runglen-\runglen,0.5*\runglen);
	\draw[hedge] (0,-0.5*\runglen) -- (\ladderlentwo*\runglen-\runglen,-0.5*\runglen);
	\foreach \i in {1,...,\ladderlentwo}{
	  \node[smallvx] (A\i) at (\i*\runglen-\runglen,0.5*\runglen){};
	  \node[smallvx] (B\i) at (\i*\runglen-\runglen,-0.5*\runglen){};
	  \draw[hedge] (A\i) -- (B\i);
	}
	\node at (0.5*\ladderlentwo*\runglen-0.5*\runglen,0.5*\runglen+0.5) {$L_2$};
	\node[smallvx,label=above:$b$] (b) at (-\offset,\runglen){};
	\node[smallvx,label=above:$c$] (c) at (\runglen*\ladderlentwo - \runglen + \offset,\runglen){};

	\end{scope}
	
	\begin{scope}[shift={(\ladderlentwo*\runglen - \runglen + 4.5*\offset,0)}]
	\draw[hedge] (0,0.5*\runglen) -- (\ladderlenthree*\runglen-\runglen,0.5*\runglen);
	\draw[hedge] (0,-0.5*\runglen) -- (\ladderlenthree*\runglen-\runglen,-0.5*\runglen);
	\foreach \i in {1,...,\ladderlenthree}{
	  \node[smallvx] (E\i) at (\i*\runglen-\runglen,0.5*\runglen){};
	  \node[smallvx] (F\i) at (\i*\runglen-\runglen,-0.5*\runglen){};
	  \draw[hedge] (E\i) -- (F\i);
	}
	\node at (0.5*\ladderlenthree*\runglen-0.5*\runglen,0.5*\runglen+0.5) {$L_3$};
	\node[smallvx,label=above:$d$] (d) at (-\offset,\runglen){};
	
	\end{scope}

	\begin{scope}[on background layer]
	\draw[bubble] (A1.center) -- (A\ladderlentwo.center) -- (B\ladderlentwo.center) -- (B1.center) -- cycle; 
	\draw[bubble] (C1.center) -- (C\ladderlenone.center) -- (D\ladderlenone.center) -- (D1.center) -- cycle; 
	\draw[bubble] (E1.center) -- (E\ladderlenthree.center) -- (F\ladderlenthree.center) -- (F1.center) -- cycle; 
	\end{scope}

	\draw[ded,out=0,in=180] (D1) to (B1);
	\draw[ded,out=0,in=180] (B\ladderlentwo) to (F1);
	
	\draw[ded,out=180,in=0] (a) to (C1);
	\draw[ded,out=0,in=180] (b) to (A1);
	\draw[ded,out=0,in=180] (A\ladderlentwo) to (c);
	\draw[ded,out=0,in=180] (d) to (E1);
	
	\end{tikzpicture}
\caption{A ladder with $\{a, b\}$ and $\{c, d\}$ on its stringers}
\label{fig:ThreeLaddersWithABCDPaths}
\end{figure}

Then $L$ has at most $11$~rungs.
\end{lemma}

\begin{proof}
Let the maximal subladders of $L$ be named as in Figure~\ref{fig:ThreeLaddersWithABCDPaths}.
We begin by observing that $L_1$ and $L_2$ can only contain at most $5$~rungs each by Lemma~\ref{tools:lem:StringersOutAtSameEndViaAandZ}.
Additionally, we get an upper bound of $3$ rungs for $L_3$ with Lemma~\ref{LadderInCondensedWall}~(iii).
Therefore, if one of $L_1$ or $L_2$ had less than $3$ rungs, or if $L_3$ had only one single rung, we get an upper bound of $11$~rungs in total, which was what we wanted.

We can therefore assume that each of $L_1$ and $L_2$ contains at least $4$ rungs, while $L_3$ must contain at least two rungs. This implies that $L_3$ is situated in a single layer of $W$ as it contains neither $a$ not $b$.
Let $R_2$ be the rung of $L_2$ which is closest to $b$.
We conclude $L_2 - R_2$ has a proper maximal subladder $L'_2$ with at least three rungs that contains neither $a$ nor $b$. It must therefore be also entirely contained in a single layer of $W$.
The same holds true for a maximal subladder $L'_1$ of $L_1 - R_1$, where $R_1$ is the rung of $L_1$ closest to $a$.
But then $L'_2$ has four disjoint paths to vertices outside of its own layer, with only one of them containing $a$ or $b$. This is a contradiction.
\qed
\end{proof}

\subsection{Proof of Theorem \ref{no14rungs}} \label{subsec:optimality}


\theoremFourteenRungs*

\begin{proof}
Consider a ladder with $l geq 14$ rungs.
As a counterexample, we use the graph~$G^*$ as introduced in the beginning of Section~\ref{sec:toolbox}.
In Lemma~\ref{tools:lem:oneDirectionOfMainProof}, we have seen that there exists no upper bound on the size of the hitting set.

It remains to show that there are no two edge-disjoint ladders in $G^*$ with at least $l$ rungs each. Suppose a ladder would contain an X-wing in $W$. Then this X-wing would contain an ($a$-$b$, $c$-$d$)-linkage in $W$. By Lemma~\ref{def:lemma:noTwoLinkages}, there are no two edge-disjoint ($a$-$b$, $c$-$d$)-linkages in $W$. Thus there can be no two edge-disjoint ladders in $G^*$ that contain an $X$ wing in $W$.

Thus it only remains to show that every ladder in $G^*$ with at least $l$ rungs must contain an X-wing in $W$.
So let $U$ be a ladder in $G^*$ with $l$~rungs.
First of all, we observe that $U$ cannot be entirely contained in $W$ due to Lemma~\ref{maxSizeInWall}. Even more clearly, it cannot be entirely contained in $G^* - W$ as there is only room for $l_A$ or $l_C$ rungs there.

We call $U_{in}$ the maximal subgraph of $U$ in $W$, while we call the one in $G^* - (W - \{a, b, c, d\})$ $U_{out}$. $U_{in}$ and $U_{out}$ are edge-disjoint, $U_{in} \cup U_{out} = U$ and both are non-empty as seen above.

\vspace{0.2cm}
\noindent \textbf{Claim 1:} \emph{$U \cap C \neq \varnothing$}
\vspace{0.2cm}

Suppose $U$ would be disjoint from $C$. Then $U$ is entirely contained in $G^* - C$.
We have seen before that in this case, it cannot be disjoint from $A$ or $W$. As $U$ is two-connected, it must therefore use both $a$ and $b$ to connect $U_{in}$ and $U_{out}$.
Now we distinguish three cases:

First, $a$ and $b$ could be on the same ladder stringer. Then either $U_{in}$ or $U_{out}$ must contain two subladders which together contain all $l$ rungs and are connected at one stringer via a path. Clearly, they cannot be entirely contained in $A$. If they were in $W$ instead, we can use Lemma~\ref{tools:lem:StringersOutAtSameEndViaAandZ} to conclude that each of them can only contain $5$ rungs, which sums up to $10 < l$~rungs in total, a contradiction.

Second, one of $U_{in}$ and $U_{out}$ might contain only (a part of) a single rung, while the other contains a ladder with $l - 1$ rungs. So many rungs cannot be contained in $A$.
If they were contained in $W$ instead, we observe that this ladder can be split into (at most) two (inclusion-)maximal subladders where both have paths connecting their stringers to $a$ and $b$. By Lemma~\ref{tools:lem:StringersOutAtSameEndViaAB}, each of those subladders can only contain up to 6~rungs. Therefore, $U_{in}$ can only contain up to $12 \leq l - 2 < l - 1$ rungs, a contradiction.

Finally, $a$ and $b$ could split $U$ into two subladders, where one is contained in $U_{in}$ and the other in $U_{out}$. Together, they must contain all $l$ rungs.
In $U_{out}$, there can be only $l_A$ rungs. In $U_{in}$, we use again Lemma~\ref{tools:lem:StringersOutAtSameEndViaAB} to see there can be only 6~rungs.
Together, this yields an upper bound of $l_A + 6 \leq l - 1 < l$ rungs for $U$, a contradiction.

As we arrived at a contradiction in all cases, we know that $U$ cannot be disjoint from $C$.
Next, we want to see that the same holds true for $A$.

\vspace{0.2cm}
\noindent \textbf{Claim 2:} \emph{$U \cap A \neq \varnothing$}
\vspace{0.2cm}

Suppose $U$ would be disjoint from $A$. Then $U$ is entirely contained in $G^* - A$.
Again, we can conclude that $U$ must use both $c$ and $d$ to connect $U_{in}$ in $W$ and $U_{out}$ in $C$.
As before, we divide the same three cases:

First, $c$ and $d$ could be on the same ladder stringer. Then $U_{in}$ or $U_{out}$ must contain two subladders $L_1, L_2$ which together contain all $l$ rungs and are connected at one stringer via a path. Again, it cannot be entirely contained in $C$. We conclude it is in $W$ instead.
We apply Lemma~\ref{tools:lem:AtMost13RungsWithCD} to see that there can be only $13$~rungs in $W$ in this case, a contradiction.

Second, one of $U_{in}$ and $U_{out}$ might contain only a single rung, while the other contains a ladder with $l - 1$ rungs. Again, so many rungs cannot be contained in $C$.
If they were contained in $W$ instead, we observe that this ladder can be split into (at most) two (inclusion-)maximal subladders $L_1, L_2$, where both have paths connecting their stringers to $c$ and $d$.

This time, we need to distinguish where $a$ and $b$ lie. As the situation is symmetric, suppose that $L_1$ contains at least as many vertices of $a$ and $b$ as $L_2$ does.

If $L_1$ contains both $a$ and $b$, we conclude that they cannot lie in $L_2$ or its paths to $c$ and $d$. Therefore, we can apply Lemma~\ref{tools:lem:StringersOutAtCD}~(i) to see that $L_2$ can only contain $2$ rungs, while $L_1$ has at most $9$ by Lemma~\ref{tools:lem:StringersOutAtCD}~(vi) in this case. This sums up to at most $11$~rungs in total.

If $L_1$ contains exactly one of $a$ and $b$ (say $a$), then $b$ might lie on $L_2$ or the paths to $c$ and $d$.
If $b$ is in $L_2$, we get at most $5$~rungs for each of $L_1$ and $L_2$ by Lemma~\ref{tools:lem:StringersOutAtCD}~(iv), which sums up to $10$~rungs in total.
If $b$ is not in $L_2$, it may lie on the paths to $c$ and $d$ for both of $L_1$ and $L_2$. Even so, we get at most $8$~rungs for $L_1$ by Lemma~\ref{tools:lem:StringersOutAtCD}~(v) and at most $4$~rungs for $L_2$ by Lemma~\ref{tools:lem:StringersOutAtCD}~(ii), which sums up to $12$~rungs in total.
In every case, we got an upper bound of $12 < l - 1$~rungs or less for $U_{in}$, a contradiction.

%
%

Finally, $c$ and $d$ could split $U$ into two subladders, where one is contained in $U_{in}$ and the other in $U_{out}$. Again, they must together contain all $l$ rungs.
In $U_{out}$, there can be only $l_C$ rungs. In $U_{in}$, we use Lemma~\ref{tools:lem:StringersOutAtCD}~(vi) to see there can be only 9~rungs.
Together, this yields an upper bound of $l_C + 9 \leq l - 1$ rungs for $U$, a contradiction. We conclude:

\begin{center}
\emph{$U$ has edges in each of $A$, $C$ and $W$.}
\end{center}

As $A$ and $C$ are different (and therefore disconnected) components of $G^* - W$, we can conclude that $U_{out}$ consists of at least two components, where at least one is in each of $A$ and $C$.

\vspace{0.2cm}
\noindent \textbf{Claim 3:} \emph{At least one of $A$ and $C$ must contain at least one edge of a rung of $U$.}

Suppose $A$ and $C$ do not contain a single edge of any rung of $U$. Our last conclusion shows that there must be some edges of $U$ in both $A$ and $C$, so those must be part of a ladder stringer. As there are only two vertices each ($a, b$ and $c, d$) separating the parts of $U_{out}$ from $U_{in}$, these parts of a ladder stringer of $U$ must each be between two adjacent rungs.

Now there are two cases:

First, the parts of the ladder stringers in $U_{out}$ could lie between the very same rungs of $U$.

Now they could lie on the same ladder stringer, meaning that $U_{in}$ contains two (inclusion-) maximal subladders with $l$ rungs in total, connected via a path at one stringer. Furthermore, each subladder has one more stringer continued via a path. One of them contains $a$ or $b$, the other contains $c$ or $d$. By Lemma~\ref{tools:lem:twoLaddersWithAZPaths}, the subladders can only have $12 \leq l - 2 < l$ rungs in total, a contradiction.

Otherwise, $a, b$ and $c, d$ must lie on different ladder stringers. Then $U_{in}$ is disconnected into two components, each containing an (inclusion-) maximal subladder that has one ladder stringer continued to one of $a$ and $b$ and the other to $c$ or $d$. The sum of rungs of both subladders is again $l$.
By Lemma~\ref{tools:lem:StringersOutAtSameEndViaAandZ}, each of those subladders can contain at most 5~rungs. This results in an upper bound of $10 \leq l - 4$ rungs for $U$, a contradiction.

Second, the ladder stringer parts in $U_{out}$ could lie between different rungs, resulting in $U_{in}$ being still connected. Moreover, $U_{in}$ will then contain exactly three (inclusion-) maximal subladders that together contain all $l$ rungs of $U$.
By Lemma~\ref{tools:lem:threeLaddersWithABCDPaths}, we get at most $11 \leq l - 3$ rungs for this collection, a contradiction.

\vspace{0.2cm}
\noindent \textbf{Claim 4:} \emph{$A$ must contain at least one edge of a rung of $U$.}
\vspace{0.2cm}

Suppose $A$ only contains part of a ladder stringer. As we have proven Claim~3, we know that $C$ must then contain some edge of a rung of $U$.
Now there are two cases:

First, $C$ might only contain part of a single rung of $U$.
Then all other $l - 1$ rungs must be contained in $U_{in}$. $U_{in}$ must then contain exactly two subladders which are connected at one stringer via a path. 
We apply Lemma~\ref{tools:lem:StringersOutAtSameEndViaAandZ} to see that each subladder contains at most $5$~rungs, which means $U_{in}$ can only contain $10 \leq l - 4 < l - 1$ rungs, a contradiction.

Second, $C$ might contain part of several rungs of $U$. This is only possible if $c$ and $d$ are situated on the ladder stringers of $U$ and $C$ contains a proper subladder of $U$. Now $U_{out}$ can contain up to $l_C$ rungs.
Then $U_{in}$ must again contain exactly two (inclusion-) maximal subladders $L_1, L_2$ which are connected via a path at one ladder stringer. This time, however, one of them (say $L_1$) also has both of its ladder stringers continued to $c$ and $d$ at the other end of $L_1$.
Moreover, each of them also has one stringer continued via a path to $a$ or $b$.

Let $L'_1$ be the (unique) maximal subladders of $L_1 - \{a, b\}$.
If $L'_1$ contains more than one rung, it is situated in a single layer of $W$. But $L'_1$ has four paths continuing its stringers, where only one of them may contain $a$ or $b$. This is a contradiction.
Therefore, $L'_1$ may only contain $1$~rung, which implies at most $2$~rungs for $L_1$. 
By Lemma~\ref{tools:lem:StringersOutAtSameEndViaAandZ}, $L_2$ has at most $5$~rungs. This sums up to at most $2 + 5 = 7$~rungs for $U_{in}$.
Together with $l_C$ rungs in $U_{out}$, we get at most $7 + l_C \leq l - 3$~rungs for $U$, a contradiction.

\vspace{0.2cm}
\noindent \textbf{Claim 5:} \emph{Both $A$ and $C$ must each contain at least one edge of a rung of $U$.}
\vspace{0.2cm}

We have seen that the claim is true for $A$. Therefore, assume it would not hold for $C$.
Similar to Claim~4, there are two cases:

First, $A$ might only contain part of a single rung of $U$.
Then all other $l - 1$ rungs must be contained in $U_{in}$. $U_{in}$ must then contain exactly two (inclusion-) maximal subladders $L_1, L_2$ which are connected at one stringer via a path.
As $a$ and $b$ are incident with the same rung $R$ of $U$ and $R$ is not a rung of $L_1$ or $L_2$, we can conclude that $a$ and $b$ are either both in $L_1$ (or both in $L_2$), but not on the rung that is incident with the paths to $c$ and $d$, or $a$ and $b$ are both not on $L_1$ or $L_2$ at all.

If $a$ and $b$ are both in $L_1$ or $L_2$, then each of $L_1$ and $L_2$ has two disjoint paths continuing their stringers to $c$ and $d$. Furthermore, none of those paths contains $a$ or $b$. (One path goes through the other ladder, and we use that $a$ and $b$ are not on the first rung.)
This implies that one of $L_1$ and $L_2$ contains at most $9$~rungs by Lemma~\ref{tools:lem:StringersOutAtCD}~(vi), while the other can contain only $2$~rungs by Lemma~\ref{tools:lem:StringersOutAtCD}~(i). This sums up to at most $11 < l - 1$~rungs for $U_{in}$, a contradiction.

If $a$ and $b$ were not in $L_1$ or $L_2$, each of the ladders con contain at most $5$~rungs by Lemma~\ref{LadderInCondensedWall}~(i). This sums up to at most $10 < l - 1$~rungs for $U$, which is again a contradiction.

Second, $A$ might contain part of several rungs of $U$, resulting in $A$ containing a proper subladder of $U$. This subladder can contain up to $l_A$ rungs.
$U_{in}$ now contains exactly two (inclusion-) maximal subladders $L_1, L_2$ which are connected at one stringer via a path. Furthermore, one of the subladders (say $L_1$) has stringers continued to $a$ and $b$.

We use the argument from the first case again to see that each of $L_1, L_2$ has a path connecting it to $c$ and another to $d$, both starting at its ladder stringers at the same end of the subladder.
Furthermore, $L_2$ cannot contain $a$ or $b$. In $L_1$, $a$ and $b$ are used to continue its stringers at the same end, so $L_1 - a - b$ contains a subladder $L'_1$ which has at most one rung less than $L_1$ and contains neither $a$ nor $b$.
We can therefore apply Lemma~\ref{tools:lem:StringersOutAtCD}~(i) to see that $L'_1$ and $L_2$ can only contain up to 2~rungs each. This means $L_1$ can only contain up to 3~rungs, which implies an upper bound of $3 + 2 = 5$ rungs for $U_{in}$.
Together with $l_A$ rungs in $U_{out}$, we arrive at an upper bound of $5 + l_A \leq l - 2$ rungs for $U$, a contradiction.

\vspace{0.2cm}

For our next claim, we pick among all possible rungs $R_A$ that are (at least partly) contained in $A$ and among all possible rungs $R_C$ that are (at least partly) contained in $C$ those two that are closest to each other (measured in the number of rungs between them).

Let $L_W$ be the subladder of $U$ that contains exactly all rungs between $R_A$ and $R_C$.
Let $L_{R_A}$ be the (inclusion-) maximal subladder of $U - L_W$ that contains $R_A$, and let $L_{R_C}$ be the (inclusion-) maximal subladder of $U - L_W$ that contains $R_C$.
Then all rungs of $U$ are contained in exactly one of $L_{R_A}, L_W$ and $L_{R_C}$.

\vspace{0.2cm}
\noindent \textbf{Claim 6:} \emph{$L_W$ contains at least 3~rungs.}
\vspace{0.2cm}

First, we will see how many rungs there can be situated in $L_{R_A}$ and $L_{R_C}$.
If $A$ contains (a part of) more than one rung, it must contain all rungs completely. Furthermore, all rungs of $L_{R_A}$ must then be in $A$, resulting in $l_A$ rungs.
Alternatively, $A$ might contain part of only one rung. Then the rest of $L_{R_A}$ must be contained in $W$. There, it contains exactly one (inclusion-) maximal subladder $L_1$ that must contain all but one rung of $L_{R_A}$. Furthermore, $L_1$ has both stringers at one end continued by one path each to $a$ and $b$.
By Lemma~\ref{tools:lem:StringersOutAtSameEndViaAB}, there can be only 6~rungs of $L_{R_A}$ in $W$, resulting in up to $6 + 1 = 7 \leq l_A$ rungs for $L_{R_A}$.

The same argumentation also holds for $C$:
$C$ might contain all rungs of $L_{R_C}$, resulting in $l_C$ rungs.
The only alternative would be that $C$ only contains part of a single rung of $L_{R_C}$. But then all other rungs of $L_{R_C}$ must be part of a single subladder $L_2$. At one end of $L_2$, both stringers are continued via one path each to $c$ and $d$, respectively.
Furthermore, $L_2$ is nor allowed to contain $a$ or $b$ as those vertices are somewhere on the border of $L_{R_A}$ and $L_W$.
We can therefore conclude that $L_2$ can only contain up to 2~rungs by Lemma~\ref{tools:lem:StringersOutAtCD}~(i). This results in an upper bound of $2 + 1 = 3 < l_C$ rungs for 
$L_{R_C}$.

Now, we can conclude that $L_{R_A}$ can only contain up to $l_A$ rungs, while $L_{R_C}$ can contain at most $l_C$ rungs.
This means there must be at least 3~rungs left for $L_W$, proving Claim~6.

\vspace{0.2cm}
\noindent \textbf{Claim 7:} \emph{$U$ contains an X-wing in $W$.}
\vspace{0.2cm}

Let $v_1, v_2$ be the endvertices of the part of $R_A$ that is contained in $A$. Similarly, let $v_3, v_4$ be the endvertices of the part of $R_C$ that is contained in $C$.
In $L_{R_A}$, $v_1$ has two (internally) disjoint paths to the ladder stringers $S_1, S_2$ of $U$. One of them crosses $v_2$, so we pick the other path $Q_1$. Say $Q_1$ connects $v_1$ to $S_1$.
Similarly, there is a path $Q_2$ on $L_{R_A}$ that does not cross $v_1$ and connects $v_2$ to $S_2$.

$S_1$ is connected to all rungs of $U$, so in particular, we can find a path $T_1$ on $S_1$ that connects $Q_1$ with the first rung of $L_W$ (the one that is closest to $L_{R_A}$). Together the paths $Q_1$ and $T_1$ form a path $P'_1$ that connects $v_1$ to a stringer of $L_W$. Thereby, $P'_1$ needs to enter $W$ via $a$ or $b$. Therefore, $P'_1$ contains $a$ or $b$.
This means that $P'_1$ contains a subpath $P_1$ in $W$ connecting $a$ or $b$ to the end of a stringer of $L_W$.
Similarly, we can find a path $P_2$ in $W$ connecting the other vertex of $a$ and $b$ to the other end of a stringer of $L_W$ (situated at the same end of $L_W$).

On the other side of $L_W$, we apply the same approach to $v_3$ and $v_4$ to find a path $P_3$ in $W$ connecting $c$ to the end of a stringer of $L_W$, and another path $P_4$ connecting $d$ to the last stringer end of $L_W$.
Note that by construction, $P_1, P_2, P_3, P_4$ and $L_W$ are internally disjoint and lie in $W$.
Together, they form an X-wing.
\qed
\end{proof}

\section{Small Ladder}\label{SmallLadder}
In this section, we will prove that the elementary ladder with three rungs has the edge-\EPP.
With that end in view, from now on, a \emph{ladder} is always supposed to be a subdivision of an elementary ladder with three rungs.

The proof goes as follows: We start with an arbitrary $2$-connected graph $G$ that contains a single vertex $v$ intersecting all ladders in said graph. We choose a tree in $G-v$ that contains all neighbours of $v$ that has certain properties. Then, if $T$ is large in some sense, we find many edge-disjoint ladders inside the union of $T$ and $v$. Otherwise, $G$ has a very simple structure (after removing some edges), which we can exploit to either find many edge-disjoint ladders or a bounded edge set intersecting all ladders.

This implies that the edge-\EPP\ holds whenever $G$ is a $2$-connected graph that contains a single vertex intersecting all ladders. We take advantage of a few useful techniques when dealing with the edge-\EPP, to see that then the edge-\EPP\ also holds in general, that is when $G$ is an arbitrary graph.

Almost the same proof can be used to prove that the \emph{house graph} (Figure~\ref{house fig}) has the edge-\EPP~as well.

\begin{figure}[!htb]
\centering
  \begin{tikzpicture}[scale=0.1, rotate=90]

			\node at (0,0) [smallvx](1){};
			\node at (0,10) [smallvx] (2) {};
			\node at (10,0) [smallvx] (3) {};
			\node at (10,10) [smallvx] (4){};
			\node at (15,5) [smallvx] (5){};

			\path[hedge]
			(1) edge node {} (2)
 				edge node {} (3)

			(4) edge node {} (2)
				edge node {} (3)
				edge node {} (5)

			(3) edge node {} (5);

	\end{tikzpicture}
  	\caption{The house graph.}\label{house fig}
\end{figure}

\subsection{Useful Techniques}
We start by proving that it is sufficient to only check the edge-\EPP\ in graphs that are $2$-connected and contain a single vertex intersecting all ladders. Furthermore, we may assume that such graphs do not contain a ladder with a bounded number of edges. Essentially, we only need to check whether the edge-\EPP\ holds in a specific subset of all graphs to deduce that it holds in general. These techniques have been used before (see for example in \cite{BH18}) and may be useful in many different settings. 

In the following, let $\mathcal{H}$ be any class of graphs. We say that $\mathcal{H}$ has the \emph{edge-{E}rd{\H o}s-{P\'o}sa property in a family of graphs $\mathcal{G}$} if there is a function $f:\mathbb{N}\to\mathbb{N}$ such that for every $k \in \mathbb{N}$ every \mbox{graph $G \in \mathcal{G}$} either contains $k$ edge-disjoint members of $\mathcal{H}$ or an edge hitting set for those of size at most $f(k)$.

If we want to show that $\mathcal{H}$ has the edge-\EPP, then the first lemma tells us that we may specify an integer $m$ at the beginning of the proof and then only check the edge-\EPP\ in graphs that do not contain a member of $\mathcal{H}$ with at most $m$ edges.

\begin{lemma} \label{smallexp}
Let $\mathcal{G}_\mathcal{H}^m$ be the family of graphs that do not contain a member of~$\mathcal{H}$ with at most $m$ edges. If $\mathcal{H}$ has the edge-\EPP\ in~$\mathcal{G}_\mathcal{H}^m$ then it already has the edge-\EPP\ in the family of all graphs. If $g$ is an edge-\EP\ function in~$\mathcal{G}_\mathcal{H}^m$ that satisfies \mbox{$g(k)\geq g(k-1)+m$} for all $k\geq 2$ then $g$ is also an edge-\EP\ function in the family of all graphs.

\end{lemma}

\begin{proof}
We know that the edge-\EPP~holds in the class $\mathcal{G}_\mathcal{H}^m$. Hence, we only need to check the edge-\EPP\ for all graphs that contain a member of $\mathcal{H}$ with at most $m$ edges. This, however, is very easy. For any such graph, just remove that member and do induction on $k$.
\qed
\end{proof}

%

With the next lemma, we may only look at graphs that contain a single vertex intersecting all members of $\mathcal{H}$, provided $\mathcal{H}$ has the vertex-\EPP. Such a vertex is also called \emph{$1$-vertex-hitting-set}.

\begin{lemma}[$1$-vertex-hitting-set \cite{BH18}] \label{1vhs}
Let $\mathcal{G}_\mathcal{H}^*$ be the family of graphs that contain a vertex that intersects all members of $\mathcal{H}$ and let $\mathcal{H}$ have the vertex-{E}rd{\H o}s-{P\'o}sa property. If $\mathcal{H}$ has the edge-{E}rd{\H o}s-{P\'o}sa property in $\mathcal{G}_\mathcal{H}^*$, then it has the edge-{E}rd{\H o}s-{P\'o}sa property in the family of all graphs. Furthermore, if $f$ is a vertex-\EP\ function and $g$ is an edge-\EP\ function in $\mathcal{G}_\mathcal{H}^*$, then $f \cdot g$ is an edge-\EP\ function in the family of all graphs.
\end{lemma}

\begin{proof}
Let $k$ be a positive integer and let $G$ be any graph.
Let $X$ be a minimal vertex hitting set in $G$. 
By induction on the size of $X$, we can show that there are either $k$ edge-disjoint subgraphs of $G$ that belong to $\mathcal{H}$ or a set of at most~$|X|g(k)$ edges that intersect all subgraphs of $G$ that belong to $\mathcal{H}$ (this is not obvious but not very hard to prove either). Note that $G$ belongs to $\mathcal{G}_\mathcal{H}^*$ if $|X|=1$ which implies the induction start.

Since the class $\mathcal{H}$ has the vertex-\EPP, there are either $k$ disjoint subgraphs of $G$ that belong to $\mathcal{H}$, which are also edge-disjoint, or a vertex hitting set of size at most $f(k)$. In the latter case together with the previous observation, we find an edge hitting set of size at most $f(k)g(k)$.
\qed
\end{proof}

The last lemma tells us that we may disregard all graphs that are not connected if all members of $\mathcal{H}$ are connected. On top of this, we may also disregard all graphs that are not $2$-connected if all members of $\mathcal{H}$ are $2$-connected.

\begin{lemma}\label{epp two conn}
For $i\in\{1,2\}$, let $\mathcal{G}_i$ be the family of $i$-connected graphs. Fix $i\in\{1,2\}$ and let $\mathcal{H}$ be a class of graphs such that each member is $i$-connected. If $\mathcal{H}$ has the edge-\EPP\ in $\mathcal{G}_i$ then it has the edge-\EPP\ in the family of all graphs. If $g$ is an edge-\EP\ function in $\mathcal{G}_i$ and either
\begin{enumerate}[label=\upshape(\roman*)]
\item it is a linear function of the form $g(k) = c(k-1)$ for some $c > 0$ or
\item it is a polynomial with non-negative coefficients such that all terms with non-zero coefficients have degree at least $2$ 
\end{enumerate} then it is also an edge-\EP\ function in the family of all graphs.
\end{lemma}

\begin{proof}
Let $\mathcal{H}$ have the edge-\EPP~in $\mathcal{G}_1$ with hitting set bound~$g$. Let $G$ be a graph with multiple components and let $k\in\mathbb{N}$. Since all members of $\mathcal{H}$ are connected, we can deal with the components separately. First, remove all components that do not contain a member of $\mathcal{H}$ as they are not needed. 
Let $\ell_j$ be the maxmimum number of edge-disjoint members of $\mathcal{H}$ in the $j$-th component of~$G$. If the sum of the $\ell_j$ is at least $k$, then we find $k$ edge-disjoint members of $\mathcal{H}$. So we may assume that the sum of the $\ell_j$ is smaller than $k$.
Now each component is $1$-connected and we can use the assumptions to find an edge-hitting set of size at most $g(\ell_j+1)$ in the $j$-th component. One can check that if $g$ satisfies one of the conditions in the statement, that the sum of $g(\ell_j+1)$ over all components is at most $g(k)$.

If all members of $\mathcal{H}$ are $2$-connected, one can do the same by looking at the blocks of $G$ instead of 
the components.
\qed
\end{proof}

Technically, as these lemmas are stated, we cannot use them together. However it is possible to apply Lemmas~\ref{epp two conn} and~\ref{smallexp} iteratively after Lemma~\ref{1vhs}.

\subsection{$(A,m)$-Trees}
Throughout this section, $A$ will be assumed to be a subset of the vertices of a graph.
By Mader's theorem~\cite{Mad78}, \emph{$A$-paths} have the edge-\EPP, that is paths whose endvertices lie in $A$ and which are otherwise disjoint from~$A$. 
These paths are trees that contain exactly two vertices of the vertex set $A$.
For the proof of the main theorem, we need to generalize this notion to trees that contain more than two vertices of the vertex set~$A$.
An \emph{$(A,m)$-tree} is a tree that contains $m$ vertices of $A$. In this section, we will prove that these trees have the edge-\EPP. The proof relies on \emph{$A$-Steiner-trees}, which are trees that contain all vertices of $A$. Kriesell \cite{Kriesell03} conjectured that if no edge set of size at most $2k-1$ separates two vertices of $A$ in a graph $G$, then there are $k$ edge-disjoint $A$-Steiner-trees in $G$. In other words, the $A$-Steiner-trees have the edge-\EPP\ with edge-\EP\ function $2k-1$. There have been multiple advances on this conjecture.
First, Lau~\cite{Lau07} proved Kriesell's conjecture when there is no edge set of size at most $26k-1$ that separates $A$. West and Wu~\cite{WW12} improved this bound to $6.5k-1$. Finally as of now, the best known bound is $5k+3$, which was proven by DeVos, McDonald and Pivotto~\cite{DeMcPi16}. 

\begin{theorem}[DeVos, McDonald and Pivotto \cite{DeMcPi16}]\label{thm:steiner}
$A$-Steiner-trees have the edge-\EPP~with edge-\EP\ function $5k+3$.
\end{theorem}

In order to prove that the $(A,m)$-trees have the edge-\EPP, a lemma is necessary.
For that we define, a \emph{mark} in a graph $G$ as a function $f:V(G)\to \{0,1\}$; a vertex $v\in V(G)$ is \emph{marked} \mbox{if $f(v)=1$}. When we decompose~$G$ into edge-disjoint subgraphs $G_1,\ldots,G_k$ such that $G=\bigcup_{i=1}^kG_i$, we say that this decomposition is \emph{$m$-good} if for $i\in[n]$, there is a mark~$f_i$ for~$G_i$ such that at least $m$ vertices in $G_i$ are marked and additionally, for any $v\in V(G)$, it holds that $\sum_{i=1}^n f_i(v)=f(v)$. The last condition means that each marked vertex of~$G$ is marked in exactly one graph~$G_i$.

Raymond, Sau and Thilikos \cite{RST16} showed that for any $m\in\N$ and any tree $T$ with sufficiently many marked vertices,
there is an $m$-good decomposition of $T$ into two subtrees. 
Through an induction, one obtains a generalization of this.

\begin{lemma} \label{edge tree decomposition}
Let $k, m \in \mathbb{N}$, let $T$ be a tree  and let $f:V(T)\to \{0,1\}$ be a mark in~$T$. Suppose there are at least $2mk$ marked vertices in $T$. There exist edge-disjoint subtrees $T_1, \ldots, T_k$ of $T$ with $T=\bigcup_{i=1}^k T_i$ such that this decomposition is $m$-good.
\end{lemma}

Note that a marked vertex of $T$ may be contained in multiple subtrees $T_i$ but it counts as a marked vertex only for one of those subtrees. 

Now we are ready to prove that the $(A,m)$-trees have the edge-\EPP.

\begin{theorem}\label{thm:AmTree}
Let $m\geq 2$ be an integer. The $(A,m)$-trees have the edge-\EPP\ with edge-\EP\ function $f(k)=2m^2k^2$.
\end{theorem}

\begin{proof}
Let $k\in\mathbb{N}$ and let $G$ be a graph and $A\subseteq V(G)$.
By Lemma \ref{epp two conn}, we may assume that~$G$ is a connected graph. 
We do induction on the size of $A$ to show that
\begin{equation}\label{AmtreeInd}
\begin{minipage}[c]{0.8\textwidth}\em
$G$ contains either $k$ edge-disjoint $(A,m)$-trees or an edge hitting set of size at most $(m-1)k|A|$ for these trees.
\end{minipage}\ignorespacesafterend 
\end{equation}

If $|A|<m$, the empty set intersects all $(A,m)$-trees since there are no $(A,m)$-trees in~$G$. Thus, the induction start is true.
Now assume that $|A|\geq m$ and let $a\in A$. For each vertex $v\in A\setminus\{a\}$, add a vertex set $A_v$ of $k$ new vertices to $G$ such that each vertex in $A_v$ is adjacent only to $v$.
Let $G'$ be the resulting graph.

Suppose there is a set $\cP$ of $(m-1)k$ edge-disjoint paths in $G'$ from $a$ to the vertices in $A^*=\bigcup_{v\in A\setminus\{a\}} A_v$. Order the vertices $v\in A$ by the number of paths that end in $A_v$, that is, let $v_1\in A$ such that the most paths of $\cP$ end in $A_{v_1}$, $v_2\in A$ the second most and so on.
By construction, at most $k$ paths of $\cP$ may end in any given set $A_v$. Therefore, there are at least $m-1$ distinct sets $A_v$ such that a path of $\cP$ ends in $A_v$. Thus, for $i\in[m-1]$, at least one path $P_i\in\cP$ ends in $A_{v_i}$. Since all paths in $\cP$ contain $a$, the union $\bigcup_{i=1}^{m-1}P_i$ is a connected graph and moreover, it contains the vertices $a,v_1,\ldots, v_{m-1}$. Any spanning tree of this union restricted to $G$ is a tree that contains at least $m$ vertices of $A$ and, therefore, contains an $(A, m)$-tree~$T_1$. 
Remove the paths $P_1,\ldots,P_n$ from $\cP$ and observe that at most $k-1$ of the remaining paths in $\cP$ end in any given 
set~$A_v$. Inductively, we can deduce that there are $k-1$ edge-disjoint $(A,m)$-trees $T_2,\ldots,T_k$ in the union of the remaining paths of $\cP$. Together with $T_1$, these are $k$ edge-disjoint $(A,m)$-trees in $G$ and we are done.

Therefore by the edge-version of Menger's theorem, there is an edge hitting set of size at most $(m-1)k-1$ for the paths from $a$ to $A^*$ in $G'$; let $X'$ be an edge hitting set of minimum size. Note that for $v\in A$ and $u\in A_v$, if $uv\in X'$, then all edges incident to $A_v$ lie in $X'$. Indeed, as $X'$ is minimum, $X'\setminus\{e\}$ is not an edge hitting set for each edge $e \in X'$. Thus, there is a path from $a$ to $u$ in $G'-(X'\setminus\{e\})$, in particular, there is a path~$P$ from~$a$ to~$v$ in~$G'-X'$. If there was an edge $f\notin X'$ incident to $A_v$, then $P\cup f$ would be a path in~$G'$ from $a$ to $A_v$ that is not met by $X'$. This would be a contradiction.

It follows that either all edges incident to a given set $A_v$ lie in $X'$ or none. As the number of edges incident to $A_v$ is $k$ and as the size of $X'$ is less than $(m-1)k$, there is a set $A'\subseteq A$ of size at most $m-2$ that contains the vertices $v$ such that the edges incident to $A_v$ are contained in $X'$. Let $X=X'\cap E(G)$; clearly, $|X|\leq |X'|$. In $G-X$, only the vertices in $A'$ may lie in the same component as $a$, however, as $|A'\cup\{a\}|<m$, there is no $(A, m)$-tree in $G-X$ that contains $a$ or any vertex in $A'$.

Apply induction on $G-X$ with vertex set $A^*=A\setminus (A'\cup \{a\})$.
There are either $k$ edge-disjoint $(A^*,m)$-trees in $G-X$ or a set $Y$ of at most $(m-1)k|A^*|$ edges that intersect all $(A,m)$-trees in $G-X$. Observe in the first case that any $(A^*,m)$-tree is an $(A,m)$-tree and in the latter case that $X\cup Y$ is an edge hitting set for the $(A,m)$-trees of size at most \mbox{$(m-1)k+(m-1)k|A^*|\leq (m-1)k|A|$.}
This concludes the proof of (\ref{AmtreeInd}).

\medskip

This almost finishes the proof of the theorem, we just need to bound the size of $A$.
Suppose the size of $A$ is at least $2mk$. 
Since $G$ is connected, there is a spanning tree~$T$ in~$G$, which contains all vertices of~$A$.
Introduce a mark $f$ on $T$ such that exactly the vertices in $A$ are marked.
Apply Lemma \ref{edge tree decomposition} to~$T$ to find an $m$-good decomposition $T_1,\ldots,T_k$ of $T$. 
Since each tree $T_i$ contains at least $m$ marked vertices, it is an $(A,m)$-tree (or rather contains one). As the trees $T_1,\ldots, T_k$ are edge-disjoint, we are done.

Therefore, we may assume that the size of $A$ is bounded by $2mk$.
Together with~(\ref{AmtreeInd}) this implies that the size of an edge hitting set may be bounded by $2m^2k^2$. This concludes the proof of the Theorem.
\qed
\end{proof}

When we choose $m=2$, this is actually a very simple proof that shows that the $A$-paths have the edge-\EPP, admittedly, with a worse edge-\EP\ function than Mader obtained.

\subsection{Graphs without Ladders (or Houses)}
In this section, we are going to characterize the graphs that do not contain a ladder.
Recall that $\Theta_r$ is the multigraph on two vertices with $r$ parallel edges between them. 
We call a graph that is a subdivision of $\Theta_r$ where each edge has been subdivided at most once a \emph{short $\Theta_r$} or just \emph{short $\Theta$}, see Figure \ref{fig:shortTheta}. Note that, as short $\Theta$'s are graphs, all but at most one edge have to be subdivided.
Let $H$ be a short $\Theta$.
The \emph{endvertices} of $H$ are the two vertices that are not inside a subdivided edge.
We call the other vertices of $H$ \emph{interior vertices} and the \emph{order} of $H$ is the number of its interior vertices, that is, $|V(H)|-2$. For technical reasons, we also call a diamond (see Figure \ref{fig:diamond}) a short $\Theta_2$, where the endvertices are the vertices of degree~$2$. 
\begin{figure}[!htb]

\subfigure{0.49\textwidth}
\centering
  \begin{tikzpicture}[scale=0.1]

\node at (10,0) [smallvx](1){};
\node at (0,10) [smallvx] (2) {};
\node at (5,10) [smallvx] (3) {};
\node at (10,10) [smallvx] (4) {};
\node at (15,10) [smallvx] (5){};
\node at (20,10) [smallvx] (6){};
\node at (10,20) [smallvx] (7){};

\path[hedge]
(1) edge node {} (2)
 	edge node {} (3)
 	edge node {} (4)
 	edge node {} (5)
 	edge node {} (6)

(7) edge node {} (2)
 	edge node {} (3)
 	edge node {} (4)
 	edge node {} (5)
 	edge node {} (6);

\path[hedge, bend right=20]
(1) edge node {} (7);

\draw[->] (30,10) [bend right=50] to (12,20);
\draw[->] (30,10) [bend left=50] to (12,0);
\node[label={endvertices}] at (41,6)  {};

\draw[->] (-7,10) to (-2,10);
\node[align=left] at (-15,9) {interior\\ vertex};
\end{tikzpicture}
  \caption{A short $\Theta_6$ of order~$5$.}\label{fig:shortTheta}
\endsubfigure
\subfigure{0.34\textwidth}
\centering
  \begin{tikzpicture}[scale=0.1]

\node at (10,0) [smallvx](1){};
\node at (5,10) [smallvx] (2) {};
\node at (15,10) [smallvx] (5){};
\node at (10,20) [smallvx] (7){};

\path[hedge]
(1) edge node {} (2)
 	edge node {} (5)

(7) edge node {} (2)
 	edge node {} (5)

(2) edge node {} (5);

\end{tikzpicture}
  \caption{The diamond graph, a short~$\Theta_2$ of order~$2$ with endvertices on top and bottom.}\label{fig:diamond}
\endsubfigure\hfill
\caption{}
\end{figure}

For $i \in [m]$, let $\Theta^i$ be a short $\Theta$ with endvertices $v_1^i$ and $v_2^i$ such that $\Theta^1, \ldots, \Theta^m$ are disjoint. We now identify $v_2^i$ and $v_1^{i+1}$ for $i\in [m-1]$ and also $v_2^m$ and $v_1^1$. We call the obtained graph a \emph{circular ordering of short $\Theta$'s}, see Figure \ref{circTheta}. 
Circular orderings of $\Theta$'s are exactly the $2$-connected graphs that do not
contain a ladder.

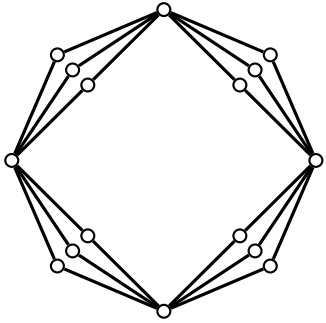
\begin{figure}[htb]
\center
\begin{tikzpicture}[scale=0.2]
\node at (10,0) [smallvx](xij){};
\node at (20,10) [smallvx](xij1){};
\node at (10,20) [smallvx](xij2){};
\node at (0,10) [smallvx](xij3){};

\node at (15,5) [smallvx](l11){};
\node at (16,4) [smallvx](l12){};
\node at (17,3) [smallvx](l13){};
\node at (14,6) [smallvx](l14){};
\node at (16,16) [smallvx](l22){};
\node at (14,14) [smallvx](l23){};
\node at (5,15) [smallvx](l31){};
\node at (4,16) [smallvx](l32){};
\node at (3,17) [smallvx](l33){};
\node at (2,18) [smallvx](l34){};
\node at (6,14) [smallvx](l35){};
\node at (5,5) [smallvx](l41){};
\node at (4,4) [smallvx](l42){};
\node at (6,6) [smallvx](l43){};

\path[hedge]
(xij) edge node {} (l11)
 	  edge node {} (l12)
 	  edge node {} (l13)
 	  edge node {} (l14)
 	  edge node {} (l41)
 	  edge node {} (l42)
 	  edge node {} (l43)

(xij1)  edge node {} (l11)
		edge node {} (l12)
		edge node {} (l13)
		edge node {} (l14)
		edge node {} (xij2)
		edge node {} (l22)
		edge node {} (l23)

(xij2)  edge node {} (l22)
		edge node {} (l23)
		edge node {} (l31)
		edge node {} (l32)
		edge node {} (l33)
		edge node {} (l34)
		edge node {} (l35)
		
(xij3)  edge node {} (l31)
		edge node {} (l32)
		edge node {} (l33)
		edge node {} (l34)
		edge node {} (l35)
		edge node {} (l41)
		edge node {} (l42)
		edge node {} (l43);

\end{tikzpicture}
\caption{A circular ordering of short $\Theta$'s.}\label{circTheta}
\end{figure}

\begin{lemma}\label{thm:noladder}
Let $G$ be a $2$-connected graph with at least six vertices.
Then $G$ does not contain a ladder if and only if $G$ is a circular ordering of short $\Theta$'s.
Additionally, if the length of a longest cycle in $G$ is at most $4$, then if $G$ does not contain a ladder, the graph $G$ is also a short $\Theta_r$.
\end{lemma}
\begin{proof}
Assume that $G$ is a circular ordering of short $\Theta$'s and suppose that $G$ contains a ladder $L$.
The ladder $L$ contains a cycle $C$ of length at least~$6$. Since the longest cycle in any short $\Theta$ has length at most~$4$, $C$ has to intersect all the endvertices of the short~$\Theta$'s in $G$. Moreover, in $L$ there is a path between two vertices at distance at least~$3$ on~$C$ and whose interior is disjoint from $C$.
However, any path that is disjoint from $C$ in $G$ connects two vertices at distance at most~$2$ on $C$ (the endvertices of a short $\Theta$ or any two vertices in a diamond). This is a contradiction, which means that $G$ does not contain a ladder.

\medskip

Now assume that $G$ is a graph on at least six vertices that does not contain a ladder.
During the proof we will find several ladders if $G$ does not satisfy certain requirements.
As $G$ does not contain any ladders, this would be a contradiction, which immediately implies that $G$ does have to satisfy those requirements.
Usually, when we say that there is a ladder, we will only specify which the two vertices of degree $3$ are. One only has to check that there are three disjoint paths between these two vertices, two of which have length at least $3$.
The following observation is useful for finding ladders.

\begin{equation}\label{cycle5path3}
\begin{minipage}[c]{0.8\textwidth}\em
A cycle $C$ of length at least $5$ together with a $V(C)$-path of length at least $3$ is a ladder.
\end{minipage}\ignorespacesafterend 
\end{equation}

\medskip

Let $C$ be a longest cycle in~$G$.
We start by proving a few structural properties of~$G$.
We say that a vertex $v$ lies outside of $C$ if $v\notin V(C)$.

\vspace{0.2cm}
\noindent \textbf{Claim 1:} \emph{All neighbours of a vertex outside of $C$ lie in $C$.}
\vspace{0.2cm}

Suppose otherwise. Let $w_1,w_2\notin V(C)$ be adjacent vertices
in~$G$. As~$G$ is $2$-connected, for $i\in[2]$, there is a path $P_i$ from $w_i$ to $c_i\in V(C)$ such that $P_1$ and $P_2$ are disjoint. The path $P=P_1 \cup w_1w_2 \cup P_2$ has length at least~$3$. If the length of $C$ is at most~$4$, then there is a path in $C$ between $c_1$ and $c_2$ of length at most $2$. Replacing this path by~$P$, yields a cycle of length at 
least~$5$, which is a contradiction to the maximality of $C$.
If the length of $C$ is at least $5$, then by~(\ref{cycle5path3}), $P\cup C$ is a ladder. This is a contradiction, which proves Claim~1.

\vspace{0.2cm}
\noindent \textbf{Claim 2:} \emph{The degree of every vertex outside of $C$ is exactly $2$. Additionally, any vertex outside of $C$ is adjacent to two vertices at distance~$2$ in $C$.}
\vspace{0.2cm}

Let $u$ be a vertex outside of $C$. As $G$ is $2$-connected, the degree of $u$ has to be at least~$2$ and by Claim~1, all of its neighbours have to be in $C$.
Let $v,w$ be two neighbours of $u$. If the distance in $C$ between $v$ and $w$ is~$1$, that is, $vw\in E(C)$, it is possible to increase the length of $C$ by additionally visiting $u$. This is a contradiction. If the distance between $v$ and $w$ is at least~$3$, 
then $C$ together with the two edges $uv,uw$ is a ladder.
Here the two vertices of degree $3$ are~$v$ and~$w$. Again, we obtain a contradiction, which implies that the distance in $C$ between each pair of neighbours of $u$ is exactly~$2$.

Suppose $u$ has three neighbours $v_1,v_2,v_3$ in $C$. Take any path $P$ in $C$ that contains these three vertices and without loss of generality assume that $v_2$ is in inbetween $v_1$ and $v_3$. Then $P$ together with $uv_1,uv_2,uv_3$ is a ladder where the vertices of degree $3$ are $u$ and $v_2$; a contradiction. Note that we use here that any path in $C$ between the neighbours of $u$ has length at least~$2$. This finishes the proof of Claim~2.

\vspace{0.2cm}
\noindent \textbf{Claim 3:} \emph{All chords of $C$ are between vertices at distance exactly $2$ in $C$.}
\vspace{0.2cm}

An edge between two vertices at distance at least~$3$ in $C$ together with $C$ is a ladder and hence, not present. This implies Claim~3.

\vspace{0.2cm}
\noindent \textbf{Claim 4:} \emph{If $v_1,v_2,v_3,v_4$ are vertices in $C$ and occur in this order in $C$ (after picking any orientation on $C$), then there are no vertices $z_1,z_2$ outside of $C$ such that $z_1$ is adjacent to $v_1$ and $v_3$ while $z_2$ is adjacent to $v_2$ and $v_4$.}
\vspace{0.2cm}

Otherwise, this would yield a ladder with $v_2$ and $v_3$ as the vertices of degree~$3$.

\medskip

Now we are ready to prove the statement of the lemma. Let $v_1v_2\ldots v_n$ be an ordering of the vertices of $C$ in some direction of $C$. Let $E$ be the set of chords of $C$. We want to show that $G'=G-E$ is a circular ordering of short $\Theta$'s.
For now we will work in $G'$. 

If there is no vertex of degree at least $3$ in $G'$, then $G'$ is a cycle and therefore, a circular ordering of short $\Theta$'s (each short $\Theta$ is a $\Theta_1$ of order $0$). 
Hence, we may assume that there is a vertex of degree~$3$,
which by Claim~2, has to be contained in $C$, say $v_1$. By Claim~2 again, the neighbourhood of $v_1$ can be partitioned into vertices that are adjacent to only $v_1$ and $v_3$ and into vertices that are adjacent to only $v_1$ and $v_{n-1}$. Without loss of generality assume that $v_1$ and $v_3$ have at least one common neighbour outside of $C$. By Claim~2 and Claim~4, $v_2$ cannot be adjacent to vertices other than $v_1$ and $v_3$ in $G'$.
Now the vertices $v_1,v_3$ and their common neighbours induce a short $\Theta$ that is separated from the rest of $G'$ by $v_1$ and $v_3$.

We proceed with $v_3$. All neighbours of $v_3$ outside of $C$ are adjacent only to $v_1$ and~$v_3$ or only to $v_3$ and $v_5$. If
$v_3$ and $v_5$ have any common neighbours, then by the same argument as before, we obtain a short $\Theta$ between $v_3$ and $v_5$ that is separated from the rest of $G'$ by $v_3$ and $v_5$. Otherwise, if $v_3$ and $v_5$ do not have any common neighbours there is a short $\Theta_1$ of order~$0$ between $v_3$ and $v_4$. 
In any case, continue with $v_5$ or $v_4$ respectively until we get back to $v_1$. Note that as $v_2$ is adjacent only to $v_1$ and $v_3$, there cannot be a short $\Theta$ such that $v_1$ is in its interior. Therefore, we obtain that $G'$ is a circular ordering of $\Theta$'s. By construction, this circular ordering contains no short $\Theta_1$ of order~$1$. 

Now we need to deal with the chords of $C$. First, suppose the length of $C$ is~$3$. Then $G=C$ since every vertex outside of $C$ would allow us to increase the length of $C$.
This means that $G$ contains fewer than six vertices, which is a contradiction. Now suppose the length of $C$ is~$4$. By Claim~2, all vertices outside of $C$ are adjacent to either $v_1$ and~$v_3$ or $v_2$ and $v_4$. By Claim~4, all vertices outside of $C$ have the same neighbourhood, say~$v_1$ and~$v_3$. As $|V(G)|\geq 6$, there is at least one vertex $u$ outside of $C$. If the 
chord~$v_2v_4$ exists, we obtain the cycle $v_1v_2v_4v_3u$ whose length is $5$, which is a contradiction to the maximality of~$C$. It follows that $G$ is a short $\Theta$ between $v_1$ and $v_3$.
Note that any short~$\Theta_r$ where $r\geq 2$ is also a circular ordering of short $\Theta$'s.

Hence, we may assume that the length of $C$ is at least $5$. Let $e$ be a chord of $C$. By Claim~3, the distance of the endvertices of $e$ in $C$ is exactly $2$, say $e=v_iv_{i+2}$. We add the edge $e$ to $G'$.
First, supppose that $v_i$ and $v_{i+2}$ are endvertices of short $\Theta$'s (possibly the same one). The vertices $v_i$ and $v_{i+2}$ cannot simultaneously also lie in the interior of a short $\Theta$ and therefore, $v_{i+1}$ is adjacent only to $v_i$ and $v_{i+2}$. After adding the edge~$v_iv_{i+2}$ to the graph, there is a short $\Theta$ between $v_i$ and $v_{i+2}$ with interior vertex $v_{i+1}$. The resulting graph still is a circular ordering of short $\Theta$'s where no short $\Theta$ is a short $\Theta_1$ of order~$1$.

Now suppose that one endvertex of $e$, say $v_i$, is a vertex in the interior of a short~$\Theta$ of order at least~$2$. By construction, there is a short $\Theta$ of order at least~$2$ between $v_{i-1}$ and $v_{i+1}$ with interior vertex $v_i$. Then there is a ladder with $v_{i-1}$ and $v_i$ as vertices of degree $3$, which is a contradiction. 

Next suppose that exactly one of the endvertices of $e$, again say $v_i$, is in the interior of a short $\Theta$ of order~$1$.
The endvertices of this short $\Theta$ are $v_{i-1}$ and $v_{i+1}$ and since this short $\Theta$ cannot be a short $\Theta_1$, the vertices $v_{i-1}$ and $v_{i+1}$ are adjacent. Moreover, $v_{i+2}$ is the endvertex of some short $\Theta$; otherwise we land in the last case, which resulted in a contradiction. Hence, there is a short $\Theta$ between $v_{i+1}$ and $v_{i+2}$, which means it is a short $\Theta_1$ of order $0$. We obtain a diamond between $v_{i-1}$ and $v_{i+2}$, which is a short $\Theta_2$. The resulting graph is still a circular ordering of short $\Theta$'s and any short $\Theta_1$ has order~$0$.

Lastly, suppose that both endvertices of $e$ are vertices in the interior of a short $\Theta$ of order $1$.
Again we deduce that $v_{i-1}$ and $v_{i+1}$ are adjacent as well as $v_{i+1}$ and $v_{i+3}$. If the length of $C$ is at least $6$, then we obtain a ladder where the vertices of degree~$3$ are $v_{i-1}$ and $v_{i+1}$ ($v_{i+1}$ and $v_{i+3}$ is also a possible choice). So suppose the length of~$C$ is exactly~$5$. The graph induced by the vertices in $C$ contains a $4$-wheel with center vertex~$v_{i+1}$ (a cycle of 
length~$4$ together with a center vertex that is adjacent to all vertices on that cycle). There is a Hamiltonian path between any two vertices in the $4$-wheel. Suppose there is a vertex $u$ outside of $C$, then $u$ has two neighbours $v,w$ on~$C$. However,
the edges $uv,uw$ together with a Hamiltonian path between $v,w$ in $G[C]$ yield a cycle of length~$6$. This is a contradiction. Now we obtain another contradiction as $|V(G)|=|V(C)|<6$.

We have now dealt with all possible cases for the chord $e$.
Inductively, add all chords of $C$ to $G'$. Observe that after adding a chord we either obtained a contradiction or the resulting graph was still a circular ordering of short $\Theta$'s where no short $\Theta$ is a short $\Theta_1$ of order~$1$.
This finishes the proof.
\qed
\end{proof}

Using a similar approach it is not that hard to see that the following characterization is also true.
We say that a graph $G$ is an \emph{$H$-expansion} if and only if it contains $H$ as a minor. 

\begin{lemma}\label{nohouse}
Let $G$ be a $2$-connected graph with at least five vertices.
Then $G$ does not contain a house-expansion if and only if $G$ is a cycle or a short $\Theta$.
\end{lemma}

\subsection{Proof of Theorem \ref{3RungLadderHasEdgeEPProperty}}
We need just one more very basic lemma which as far as we know, has not been written down anywhere else. A \emph{segment} of a tree $T$ is a non-trivial path $P$ in~$T$ between two vertices whose degree in~$T$ is distinct from~$2$ and such that all the interior vertices of $P$ have degree~$2$ in~$T$.

\begin{lemma}\label{lem:segment}
The number of segments of a tree is at most $2$ times the number of its leaves.
\end{lemma}

\begin{proof}
Through a simple induction, we obtain that the number of branch vertices in a tree, that is, vertices of degree at least $3$, is bounded from above by the number of its leaves.
Let~$T$ be a tree with $\ell$ leaves and $r$ branch vertices; it follows that $r\leq\ell$. Clearly, we may assume that $T$ contains at least two vertices and thus, a segment.
Contract all segments of $T$ to single edges and denote by $T'$ the resulting tree. Observe that the number of segments of $T$ coincides with the number of edges in $T'$ and that the vertex set of $T'$ is exactly the set of leaves and branch vertices of $T$. We deduce that
\mbox{$|E(T')|=|V(T')|-1=r+\ell-1<2\ell.$}
\qed
\end{proof}

Finally, we can prove that the house graph and the ladder with three rungs have the edge-\EPP. As there is a significant overlap in the two proofs, we start with a common lemma.

\begin{lemma}\label{pretree-lemma}
Let $H$ be either the house graph or the ladder with three rungs. Let $G$ be a {$2$-connected} graph that contains a vertex $v$ that intersects all $H$-expansions and each $H$-expansion contains at least eight edges. Additionally, assume that there are no $k$ edge-disjoint $H$-expansions in $G$. Then there is an edge set $E \subseteq E(G)$ of size at most $104k$ such that in $G-E$ all neighbours of~$v$ are of degree~$2$. Moreover, there are at most $6k$ vertices at distance~$2$ from~$v$.
\end{lemma}

\begin{proof}
Note that, whenever we find $k$ edge-disjoint ladders or a ladder with less than eight edges, that these are also $k$ edge-disjoint house-expansions or a house-expansion with less than eight edges.

For a tree $T$, define a \emph{pre-leaf} as a vertex that is adjacent to at least one leaf. The \emph{order} of a pre-leaf is the number of leaves it is adjacent to. Let $P_T$ be the set of pre-leaves and $P^1_T\subseteq P_T$ the set of pre-leaves of order~$1$ in $T$. 
A leaf~$\ell$ \emph{belongs} to the unique pre-leaf~$p$ it is adjacent to. We also say that \emph{$\ell$ is a leaf of~$p$}.

Since $G$ is $2$-connected, we immediately obtain that $G-v$ is connected and therefore, we know that there is a tree in $G-v$ that contains all neighbours of $v$ and such that all of its leaves are neighbours of $v$. Indeed, take any spanning tree of $G-v$ and gradually remove leaves of this tree that are not neighbours of~$v$. Consider all trees in $G-v$ with this property and let $T$ be the one that maximizes $|P_T|$ and then minimizes $|P^1_T|$. First we want to bound the size of $P_T$.

\vspace{0.2cm}
\noindent \textbf{Claim 1:} \emph{The size of $P_T$ is bounded by $6k$, that is, $|P_T|\leq 6k$.}
\vspace{0.2cm}

Suppose otherwise. Define a mark on $T$ such that exactly the vertices in $P_T$ are marked; there are at least $2\cdot 3k$ marked vertices in $T$. Use Lemma \ref{edge tree decomposition} to 
find a $3$-good partition of $T$ into $k$ edge-disjoint trees $T_1,\ldots, T_k$. Thus, there are distinct marked vertices $a_i^j\in V(T)$ where $i\in[3]$ and $j\in[k]$ such that $a_i^j$ is marked in $T_j$.

For $j\in [k]$, let $P$ be the unique path in $T_j$ from $a_1^j$ to $a_2^j$ and let $Q$ be the unique path in $T_j$ from $a_3^j$ to $V(P)$. Let the endvertex of $Q$ on $P$ be $r$. By construction, the paths in $T_j$ from $r$ to $a_1^j$, to $a_2^j$ and to $a_3^j$ are disjoint outside of $r$. At least two of those paths have length at least $1$ since the vertices $a_1^j,a_2^j$ and $a_3^j$ are distinct, say the paths from $r$ to $a_1^j$ and $a_3^j$. 
As each marked vertex in $T$ is a pre-leaf,
each vertex $a_i^j$ is adjacent to a leaf of $T$.
The graph induced by the vertices $V(P),V(Q)v$ and one leaf of each vertex~$a_i^j$ is a ladder, see Figure \ref{fig1}. Recall that all leaves of $T$ are neighbours of $v$.

\begin{figure} [htb]
\center
\begin{tikzpicture}[scale=0.15]

\node at (10,-5) [smallvx](v){};
\node at (0,5) [smallvx] (l1) {};
\node at (10,0) [smallvx] (l2) {};
\node at (20,5) [smallvx] (l3){};
\node at (0,10) [smallvx] (a1){};
\node at (10,10) [smallvx] (r){};
\node at (20,10) [smallvx] (a3){};
\node at (10,5) [smallvx] (a2){};

\path[hedge]
(v) edge node {} (l1)
 	edge node {} (l2)
 	edge node {} (l3)

(a1) edge node {} (l1)

(a2) edge node {} (l2)
	 
(a3) edge node {} (l3);	

\draw [decorate,decoration=snake] (r) -- (a1);
\draw [decorate,decoration=snake] (r) -- (a2);
\draw [decorate,decoration=snake] (r) -- (a3);

\node at(12,-5){\textbf{$v$}};
\node at(0,12.5){\textbf{$a_1^j$}};
\node at(10,12.5){\textbf{$r$}};
\node at(20,12.5){\textbf{$a_3^j$}};
\node at(8,5){\textbf{$a_2^j$}};

\end{tikzpicture}

\caption{This graph is a ladder even if $r=a_2^j$.}
\label{fig1}
\end{figure}

As the vertices $a_i^j$ are distinct, also the leaves of $T$ that belong to them are distinct. Since the trees $T_1,\ldots, T_k$ are edge-disjoint, we obtain $k$ edge-disjoint ladders from these $k$ trees, which is a contradiction. This proves Claim~1.

\vspace{0.2cm}
\noindent \textbf{Claim 2:} \emph{For $n\in\N$, if $P$ is a path in $G-v$ that contains $5n$ neighbours of $v$, then there are $n$ ladders in $G[P\cup \{v\}]$.}
\vspace{0.2cm}

Let $p_1, ..,p_5$ be the first five neighbours of $v$ on $P$ (starting in any endvertex of $P$). Let $Q=p_1Pp_{5}$. In $G[Q\cup \{v\}]$ there is a ladder, as seen in Figure \ref{manypath}. We proceed with the next five neighbours of $v$ on $P$ to find another ladder that is edge-disjoint from the first one since they intersect only in $v$. As there are $5n$ neighbours of $v$ on $P$, we inductively obtain $n$ edge-disjoint ladders in $G[P\cup \{v\}]$, which proves Claim~2.

\begin{figure}[htb]
\center
\begin{tikzpicture}[scale=0.3]

\node at (8,0) [smallvx](v){};
\node at (0,5) [smallvx] (p1) {};
\node at (4,5) [smallvx] (p2) {};
\node at (8,5) [smallvx] (p3){};
\node at (12,5) [smallvx] (p4){};
\node at (16,5) [smallvx] (p5){};

\path[hedge]
(v) edge node {} (p1)
 	edge node {} (p3)
    edge node {} (p5);	
    
\draw [decorate,decoration=snake] (p1) -- (p2);
\draw [decorate,decoration=snake] (p2) -- (p3);
\draw [decorate,decoration=snake] (p3) -- (p4);
\draw [decorate,decoration=snake] (p4) -- (p5);

\node at(0,6){\textbf{$p_1$}};
\node at(4,6){\textbf{$p_2$}};
\node at(8,6){\textbf{$p_3$}};
\node at(12,6){\textbf{$p_4$}};
\node at(16,6){\textbf{$p_5$}};
\node at(8,-1){\textbf{$v$}};

\end{tikzpicture}
\caption{A ladder where $v$ and $p_3$ are the vertices of degree $3$.}
\label{manypath}
\end{figure}
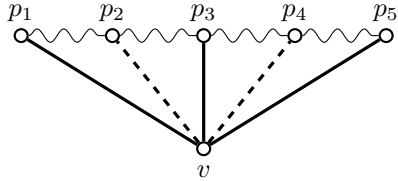

\vspace{0.2cm}
\noindent \textbf{Claim 3:} \emph{The number of neighbours of $v$ that are not leaves of $T$ is less than~$80k$.}
\vspace{0.2cm}

Let $T'$ be the subtree of $T$ that we obtain by removing all leaves of $T$.
We immediately deduce that all leaves of $T'$ are pre-leaves of $T$. Hence by Claim~1, the number of leaves of $T'$ is bounded by $6k$ and by Lemma \ref{lem:segment}, the tree $T'$ contains at most $12k$ segments. If there are any edges from $v$ to endvertices of segments of $T'$, remove them; these are at most $24k$ edges. 
For each segment $S$ of $T'$, let $\ell_S$ be the number of remaining neighbours of $v$ on $S$ and let $n_S=\lfloor\frac{\ell_S}{5}\rfloor$ be the number of full sets of five neighbours of $v$ on $S$.
If the sum of the $n_S$ over all segments $S$ is at least $k$, then applying Claim~2 to each segment of~$T$ yields at least $k$ ladders. These ladders are edge-disjoint since no two segments share a neighbour of $v$.
Therefore, we may assume that the sum of the $n_S$ is smaller than $k$.
The number of edges from $v$ to $T'$ is at most $$24k+\sum_{\substack{\text{$S$ segment}\\ \text{of $T'$}}}\ell_S\leq 24k+\sum_{\substack{\text{$S$ segment}\\ \text{of $T'$}}}(5n_S+4)< 24k+5k+4\cdot 12k<80k.$$
As $T'$ contains all neighbours of $v$ that are not leaves of $T$, this proves Claim~3.

\medskip

Let $E_1$ be the set of edges from $v$ to vertices of $T$ that are not leaves. From now on assume that $E_1$ has been removed from $G$. All remaining edges incident to $v$ are incident to leaves of $T$.
Note that $G-E_1$ is still $2$-connected. 
Indeed, the graph $G[T\cup \{v\}]$ is $2$-connected and so is $G[T\cup \{v\}]-E_1$. Then also $G-E_1$ has to be $2$-connected as $E_1\subseteq  G[T\cup \{v\}]$. 

The next thing we want to prove is that leaves of $T$ belonging to pre-leaves of order at least $4$, collectively, do not have many neighbours in $G$.

\vspace{0.2cm}
\noindent \textbf{Claim 4:} \emph{There are at most $6k$ leaves of $T$ belonging to pre-leaves of order at least $4$ that are adjacent to vertices other than $v$ and their pre-leaf.}
\vspace{0.2cm}

First we show that:
	\begin{equation}\label{endin1}
	\begin{minipage}[c]{0.8\textwidth}\em
If a leaf~$\ell$ of a pre-leaf~$w$ of order at least~$4$ has a neighour~$x\notin\{v,w\}$, then there is a path from~$x$ to $T$ that does not meet~$\ell$. All paths from $x$ to $T$ that are disjoint from $\ell$ end in a leaf of a pre-leaf of order $1$ and do not intersect~$v$.
	\end{minipage}\ignorespacesafterend 
	\end{equation} 

We start by proving that there is a path from $x$ to $T$ that is disjoint from~$\ell$. Clearly, we are done if $x \in V(T)$ already. If $x \notin V(T)$, by $2$-connectedness of $G-E_1$, there are two paths from $x$ to $T$ that meet only in $x$ and, hence, at most one of them may contain~$\ell$. 
Note that no path from $x$ to $T$ may intersect $v$ as all neighbours of $v$ lie in $T$.

Now let $P$ be a path from $x$ to a vertex $t\in V(T)$ that does not contain~$\ell$.
Add  $P$ and the edge~$x\ell$ to~$T$ and remove the edge $w\ell$. It is easy to see that this yields again a tree~$T'$ that contains all neighbours of $v$ such that all its leaves are neighbours of $v$. Note that $w$ is still a preleaf in $T'$ as it had leaves other than $\ell$. We want to prove that $t$ has to be a leaf of a pre-leaf of order~$1$ by doing a case analysis, which then proves (\ref{endin1}).

If $t=x$ was a preleaf, then we find a ladder with at most $7$ edges, see Figure~\ref{smallladder}. This is a contradiction. 
\begin{figure}[htb]
\center
\begin{tikzpicture}[scale=0.2]
\node at (5,0) [smallvx](v){};
\node at (0,5) [smallvx] (l1) {};
\node at (5,5) [smallvx] (l2){};
\node at (10,5) [smallvx] (l3){};
\node at (2.5,10) [smallvx] (x){};
\node at (7.5,10) [smallvx] (w){};

\path[hedge]
(v) edge node {} (l1)
 	edge node {} (l2)
 	edge node {} (l3)

(x) edge node {} (l1)
	edge node {} (l2)

(w) edge node {} (l2)
	edge node {} (l3);

\node at(5,-1.5){\textbf{$v$}};
\node at(5,7){\textbf{$\ell$}};
\node at(2.5,11.5){\textbf{$x$}};
\node at(7.5,11.5){\textbf{$w$}};
\end{tikzpicture}
\caption{This is the ladder with seven edges we find.}
\label{smallladder}
\end{figure}
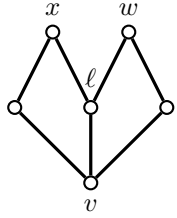

Now suppose that $t$ is not a leaf of $T$ and suppose that $x$ is no pre-leaf of~$T$.
The vertex $x$ now becomes the pre-leaf of $\ell$ in $T'$ while it was no pre-leaf in~$T$.
Note that all pre-leaves of $T$ are still pre-leaves of $T'$ and hence, $T'$ has more pre-leaves than $T$, which is a contradiction to the choice of $T$.

Lastly, suppose that $t$ is a leaf of $T$ that belongs to a pre-leaf $w_0$ of order at least $2$.
The vertex $w_0$ is still a pre-leaf in $T'$ (even if $w_0=w$) while $x$ again becomes a pre-leaf of $T'$. Hence, $T'$ has more pre-leaves than $T$, which is a contradiction. This finishes the proof of (\ref{endin1}). 

\medskip

The next thing we want to do is to bound the number of paths that go into a leaf of a pre-leaf of order~$1$. The paths that we look at are constructed as in~(\ref{endin1}).

\begin{equation}\label{notwopath}
		\begin{minipage}[c]{0.8\textwidth}\em
There are no two distinct leaves $\ell_1, \ell_2$ of pre-leaves of order at least $4$ such that there is a leaf $\ell^*$ of a pre-leaf $w_0$ of order~$1$ and for $i\in [2]$, a path $P_i$ from $\ell_i$ to $\ell^*$ that is internally disjoint from $T$ and $v$.
		\end{minipage}\ignorespacesafterend 
		\end{equation}

Suppose otherwise. For $i\in [2]$, let $w_i$ be the pre-leaf of $\ell_i$. Follow $P_1$ from $\ell_1$ until it intersects $P_2$ for the first time in a vertex $x$. Add~$P_2$ and $\ell_1P_1x$ to~$T$ and remove the edges $\ell_1w_1$ and $\ell_2w_2$ from~$T$. Let $T'$ be the resulting tree and observe that all neighbours of $v$ lie in $T'$ and all leaves of $T'$ are neighbours of $v$. The vertex $w_0$ is not a pre-leaf anymore in $T'$ while $w_1$ and $w_2$ are still pre-leaves of order at least~$2$ in $T'$ since their order in $T$ was at least $4$ (even if $w_1=w_2$). If $x$ is adjacent to both $\ell_1$ and $\ell_2$, then $x$ becomes a pre-leaf of order~$2$ in $T'$. Now the tree $T'$ has the same number of pre-leaves but fewer pre-leaves of order~$1$. 

So we may assume that $x$ is not adjacent to both $\ell_1$ and $\ell_2$. For $i\in[2]$, let $p_i$ be the unique neighbour of~$\ell_i$ on~$P_i$. It follows that $p_1 \neq p_2$ and that both of these vertices become pre-leaves in $T'$.
Then the number of pre-leaves in $T'$ is larger than in $T$. This is a contradiction and thus, proves (\ref{notwopath}).

\medskip

By Claim~1, there are at most $6k$ pre-leaves of order~$1$ in $T$ and, thus, at most $6k$ leaves belonging to pre-leaves of order~$1$. Combining this with (\ref{endin1}) and (\ref{notwopath}), we obtain that there are at most $6k$ leaves of pre-leaves of order at least $4$ that are adjacent to vertices other than $v$ and their pre-leaf, thus, proving Claim~4. 

\medskip

Let $E_2$ be the set of edges from $v$ to leaves belonging to pre-leaves of order at least $4$ that are adjacent to vertices other than $v$ and their pre-leaf. And let $E_3$ be the set of edges from $v$ to leaves belonging to pre-leaves of order at most $3$. It holds that $|E_2|\leq 6k$ and $|E_3|\leq 3\cdot 6k$. 
Let $E = E_1 \cup E_2 \cup E_3$. Then it holds that $|E| \leq 104k$ and all neighbours of $v$ in $G-E$ are adjacent to $v$ and its preleaf, that is, they are of degree~$2$. Since the vertices at distance~$2$ from $v$ are the preleaves of $T$, by Claim~1, there are at most $6k$ such vertices. This proves Lemma~\ref{pretree-lemma}.
\qed
\end{proof}

We start by proving that the house graph has the edge-\EPP.

\begin{theorem}
The house graph has the edge-\EPP.
\end{theorem}

\begin{proof}
In \cite{BHJR19} it was shown that there is a vertex-\EP\ 
function~$g$ for the ladders that lies in $\mathcal{O}(k\log k)$. Let $f_0(k)=236(k-1)$. We prove that \mbox{$f=g\cdot f_0\in\mathcal{O}(k^2\log(k))$} is an edge-\EP\ function for the house-expansions. Let $k \in \mathbb{N}$ and let $G$ be a graph. We assume that $G$ is $2$-connected, contains a vertex $v$ intersecting all ladders and no ladder in $G$ contains less than eight edges.
To finish the proof, by Lemmas~\ref{smallexp}, \ref{1vhs} and \ref{epp two conn}, it is sufficient to show that $G$ contains either $k$ edge-disjoint ladders or at most $f_0(k)$ edges that intersect all ladders. 
Note that $f_0(k)\geq f_0(k-1)+7$ and $f_0$ is of the form $c(k-1)$ for $c>0$. 
In the following, we assume that there are no $k$ edge-disjoint ladders in $G$. Clearly, we would be done otherwise.

Now, by Lemma~\ref{pretree-lemma}, we can remove an edge set $E_1$ of size at most $104k$ such that all neighbours of $v$ are of degree~$2$. We call the vertices at distance~$2$ from $v$ \emph{majors}. For each major that has exactly one common neighbour with~$v$, remove the edge from $v$ to that common neighbour. Let $E_2$ be the set of these edges. Note that by the lemma, the size of $E_2$ is bounded by $6k$. 
Let $w_1, \ldots, w_n$ be the remaining majors. Each major has at least two common neighbours with $v$.
Split~$v$ into vertices $v_1,\ldots, v_n$ such that $v_i$ is adjacent to all common neighbours of $w_i$ and $v$ and let \mbox{$A_v=\{v_1, \ldots, v_n\}$.} 

Suppose $\cP = \{P_1,\ldots, P_{4k}\}$ is a set of $4k$ edge-disjoint $A_v$-paths .
Observe that any $A_v$-path contains exactly two neighbours of $v$ and no two edge-disjoint $A_v$-paths can meet in a neighbour of $v$ as their degree is~$2$. Let the endvertices of $P_i$ be $v_{i_1}$ and $v_{i_2}$. 
Identify the vertices $v_{i_1}$ and $v_{i_2}$ again and add to the resulting cycle in $G$ one common neighbour $x$ of $v$ and $w_{i_1}$ that does not lie in $P_i$ and both its edges to $v$ and $w_{i_1}$. Such a common neighbour exists as $w_{i_1}$ has at least two common neighbours with $v$. Remove the edges of $P_i$ and the vertex $x$ and if after that, there is at most one common neighbour between $w_{i_1}$ and $v$ or between $w_{i_2}$ and $v$, also remove these common neighbours. Then remove $P_i$ from $\cP$ and any path that intersects one of the removed common neighbours. These are at most $4$ paths. 
Note that any remaining major still has at least two common neighbours with $v$. Thus, we may continue until we construct $k$ house-expansions.

Hence, we may assume that there are no $4k$ edge-disjoint $A_v$-paths. By Mader's theorem, we can remove a set $E_3$ of at most $8k-8$ such that there are no $A_v$-paths anymore. Then any remaining major lies in its own block (they are connected only through $v$). And, thus, any block that contains $v$ is an edge or a short $\Theta$ (it contains $v$, a major and their common neighbours). No such block contains a house-expansion by Lemma~\ref{nohouse} and as every house-expansion has to intersect $v$, we found a hitting set. The size of this hitting set is $|E_1 \cup E_2 \cup E_3| \leq 118k$. 

For $k\geq 2$, it holds that $f_0(k) \geq 118k$ and for $k=1$, it holds that a hitting set of size $0$ suffices. Hence, $f_0$ is an edge-\EP\ function in the case defined at the beginning of this proof. This finishes the proof.

\qed
\end{proof}

Now we finish this section by proving that the ladder with three rungs has the edge-\EPP.

\begin{proof}[Proof of Theorem \ref{3RungLadderHasEdgeEPProperty}]
We start in the same way as the proof for the house graph. We obtain a vertex-\EP\ function
$g \in \mathcal{O}(k\log k)$ from \cite{BHJR19}. For $f_0(k)=131k^2$, we will prove that
\mbox{$f=g\cdot f_0\in\mathcal{O}(k^3\log(k))$} is an edge-\EP\ function for the ladders. Let $k \in \mathbb{N}$ and let $G$ be a graph. We again assume that $G$ is $2$-connected, contains a vertex $v$ intersecting all ladders and no ladder in $G$ contains less than eight edges.
If we can prove that $G$ contains either $k$ edge-disjoint ladders or at most $f_0(k)$ edges that intersect all ladders, then by Lemmas~\ref{smallexp}, \ref{1vhs} and \ref{epp two conn}, we are done. 
Note that $f_0(k)\geq f_0(k-1)+7$ and $f_0$ is a polynomial where all coefficients are non-negative and the terms with positive coefficients are of degree at least~$2$.
In the following, we assume that there are no $k$ edge-disjoint ladders in $G$. Clearly, we would be done otherwise.

By applying Lemma~\ref{pretree-lemma}, we obtain an edge set $E_1$ of size at most $104k$ such that in $G-E_1$ all neighbours of $v$ are of degree~$2$. We call the vertices at distance~$2$ from $v$ \emph{majors}. Let $w_1, \ldots, w_n$ be the majors in $G-E_1$. Split~$v$ into vertices $v_1,\ldots, v_n$ such that $v_i$ is adjacent to all common neighbours of $w_i$ and $v$ and let \mbox{$A_v=\{v_1, \ldots, v_n\}$.} Let~$T_0$ be an $(A_v,3)$-tree that contains the vertices $v_{i_1},v_{i_2},v_{i_3}\in A_v$.
Identify the vertices $v_{i_1},v_{i_2},v_{i_3}$ again and see that we obtain a ladder from~$T_0$. Indeed, as all neighbours of $v$ are adjacent only to $v$ and a major, the majors $w_{i_1},w_{i_2},w_{i_3}$ have to be contained in $T_0$. Then we obtain a ladder in the same way as in Claim~1 of Lemma~\ref{pretree-lemma} in Figure \ref{fig1}. Now apply Theorem \ref{thm:AmTree} to $G-E_1$ to either find $k$ edge-disjoint $(A_v,3)$-trees or a set $E_2$ of at most $18k^2$ edges that intersects all such trees. In the first case we immediately obtain $k$ edge-disjoint ladders, which is a contradiction. Thus, identify all vertices in~$A_v$ to the vertex $v$ again and let $G'=G-E_1-E_2$.

\vspace{0.2cm}
\noindent \textbf{Claim:} \emph{For each $\ell \in \mathbb{N}$, in every block of $G'$ there are either $\ell$ edge-disjoint ladders or a set of at most $5\ell-1$ edges that intersects all ladders.}
\vspace{0.2cm}

Assuming this claim is true, we want to show that this proves the theorem.
As every ladder in $G'$ has to be entirely contained in a block of $G'$ and as blocks are edge-disjoint, we can look at each block separately.
First, remove each block that does not contain a ladder. Thus, as each ladder contains $v$, we remove all blocks that do not contain $v$.
For each of the remaining blocks $B_1,\ldots, B_n$, count the maximum number of edge-disjoint ladders $\ell_1,\ldots, \ell_n\geq 1$. If the sum of the $\ell_i$ exceeds $k-1$, there are $k$ edge-disjoint ladders in $G'\subseteq G$ and we are done. Hence, we may assume that the sum of the $\ell_i$ is at most~$k-1$. Then as $\ell_i\geq 1$ for $i\in[n]$, it holds that $n<k$. By the Claim, we can find a set of at most $5(\ell_i+1)-1$ edges that intersects all ladders in $B_i$. Altogether, there is an edge set $E_3$ of size at most $\sum_{i=1}^n (5\ell_i+4)\leq 5(k-1)+4n \leq 9(k-1)$ that intersects all ladders in $G'$. 
Then $E_1\cup E_2\cup E_3$ is an edge hitting set for the ladders in~$G$ of size at most $18k^2+113k\leq f_0(k)$, which concludes the proof.

\bigskip

Let us prove Claim~5 now. Clearly, the statement is true if $\ell=1$, so we may assume that $\ell\geq 2$.
Any block of $G'$ that does not contain a ladder, contains a set of zero edges that intersects all ladders. Thus, the claim is true in this case as well. 
Hence, let $B$ be a block of $G'$ that contains a ladder and therefore, $B$ also has to contain~$v$. Let $\mathcal{M}$ be the set of majors in $G'$. 

For any major $w\in \mathcal{M}$, there is a short $\Theta$ between $v$ and $w$ that is separated from the rest of $G'$ by $v$ and $w$. This is the beginning of a circular ordering of short $\Theta$'s. Remember that by Lemma~\ref{thm:noladder}, circular orderings of short $\Theta$'s do not contain ladders.

We claim that:
\begin{equation}\label{twopre-leaves}
		\begin{minipage}[c]{0.8\textwidth}\em
there are exactly two vertices of $\mathcal{M}$ that are contained in $B$.
		\end{minipage}\ignorespacesafterend 
		\end{equation}

If no vertex of $\mathcal{M}$ lies in $B$, then $B$ is just an edge. Clearly, then $B$ does not contain a ladder, which is a contradiction. If there is exactly one vertex $w\in \mathcal{M}\cap V(B)$, then $B$ is a short $\Theta$ between $v$ and $w$. Again, we obtain a contradiction since there is no ladder in a short $\Theta$.
Lastly, suppose there are at least three vertices of $\mathcal{M}$ in $B$. 
Let $T_0\subseteq B$ be a tree in $B-v$ that contains three vertices of $\mathcal{M}$; such a tree exists by $2$-connectedness of~$B$. Each of these three vertices has a common neighbour with $v$ in $B$. This, however, means that when we split $v$ into the vertices of $A_v$ again, we obtain an $(A_v,3)$-tree that is disjoint from $E_2$. This is a contradiction and proves (\ref{twopre-leaves}).

\medskip

Let $w_1,w_2$ be the two vertices in $\mathcal{M}\cap V(B)$. As $B$ is $2$-connected and as all neighbours of $v$ are of degree~$2$, the subgraph $B'\subseteq B$ that results from removing $v$ and its neighbours from $B$ is connected. We call the blocks of $B'$ and also the two short $\Theta$'s in $B$ between $v,w_1$ and also between $v,w_2$, the \emph{parts} of $B$. Observe that no part contains a ladder. Indeed, any part that does not contain $v$ cannot contain a ladder. On the other hand, the parts that contain $v$ are short $\Theta$'s and hence, do not contain a ladder. 

By $2$-connectedness of $B$, the parts of $B$ can be ordered cyclically as seen in Figure~\ref{parts}, that is, every part can be separated from the rest of $B$ by two vertices that we call \emph{passing points} of that part. Every passing point is contained in exactly two parts. Note that the passing points of $B$ are $v,w_1,w_2$ and the cutvertices of $B'$.

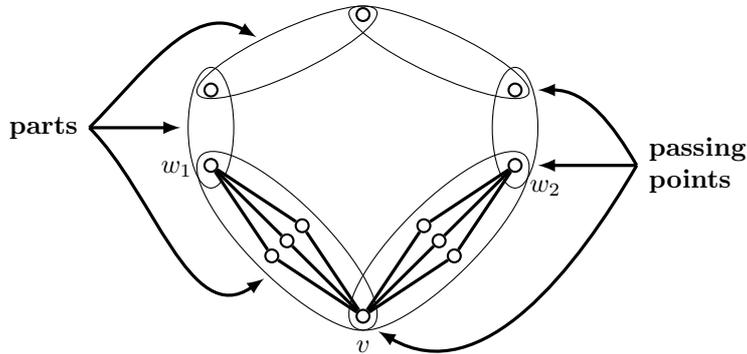
\begin{figure}[htb]
\center
\begin{tikzpicture}[scale=0.2]
\node at (10,0) [smallvx](v){};
\node at (0,10) [smallvx](w1){};
\node at (20,10) [smallvx](w2){};

\node at (4,4) [smallvx](l1){};
\node at (5,5) [smallvx](l2){};
\node at (6,6) [smallvx](l3){};

\node at (16,4) [smallvx](l4){};
\node at (15,5) [smallvx](l5){};
\node at (14,6) [smallvx](l6){};

\node at (20,15) [smallvx](p3){};
\node at (0,15) [smallvx](p1){};
\node at (10,20) [smallvx](p2){};

\node at (10,-2) (x1){\textbf{$v$}};
\node at (-2.25,9.75) (x2){\textbf{$w_1$}};
\node at (22,8.5) (x3){\textbf{$w_2$}};
\node[align=left] at (32,10) (x4){\textbf{passing}\\ \textbf{points}};
\node at (-11,12.5) (x5){\textbf{parts}};
\coordinate (q1) at (28,10);
\coordinate (q2) at (-8,12.5);

\path[->, very thick] (q1) edge node {} (21.5,10);
\path[->, very thick, in=330, out=240] (q1) edge node {} (11,-1);
\path[->,very thick, in=0, out=120] (q1) edge node {} (21.5,15);

\path[->, very thick] (q2) 	edge node {} (-2,12.5);
\path[->,very thick, in=150, out=45] (q2) edge node {} (3,18.5);
\path[->, very thick, in=215, out=315] (q2) edge node {} (3.5,2.5);

\path[hedge]
(v) 	edge node {} (l1)
 	  	edge node {} (l2)
 	  	edge node {} (l3)
 	  	edge node {} (l4)
 	  	edge node {} (l5)
 	  	edge node {} (l6)
 	  	
(w1) 	edge node {} (l1)
 	  	edge node {} (l2)
 	  	edge node {} (l3)
 	  	
(w2)  	edge node {} (l4)
 	  	edge node {} (l5)
 	  	edge node {} (l6);
 	  	
\draw[rotate around={315:(5,5)}] (5,5) ellipse (8 and 2.5);
\draw[rotate around={45:(15,5)}] (15,5) ellipse (8 and 2.5);
\draw[rotate around={90:(20,12.5)}] (20,12.5) ellipse (4 and 1.5);
\draw[rotate around={335:(15,17.5)}] (15,17.5) ellipse (6.5 and 1.5);
\draw[rotate around={25:(5,17.5)}] (5,17.5) ellipse (6.5 and 1.5);
\draw[rotate around={90:(0,12.5)}] (0,12.5) ellipse (4 and 1.5);

\end{tikzpicture}
\caption{This shows the parts and passing points.}\label{parts}
\end{figure}

Suppose there are two passing points $p_1,p_2$ belonging to the same part $R$ such that there are fewer than $5\ell$ edge-disjoint paths between $p_1$ and $p_2$ in $R$. By the edge-version of Menger's theorem, there is a set $E$ of at most $5\ell-1$ edges that intersects all paths in~$R$ between~$p_1$ and~$p_2$. 
Because of the cyclical structure of $B$, every block of $B-E$ is a subgraph of a part of $B$. However, we already observed before that no part of $B$ contains a ladder. Hence, $E$ is an edge hitting set for the ladders in $B$ of size at most~$5\ell-1$; we are done.

So we may assume that in any part there are at least $5\ell$ edge-disjoint paths between its two passing points. As $\ell\geq 2$, there are at least six vertices in each part and each passing point has more than three neighbours in any part that it lies in. This implies that the parts containing $v$ are short $\Theta$'s of order at least~$2$, which means they are $2$-connected. Clearly, each part of $B$ that is a block of $B'$ is $2$-connected as well.
Hence, we are able to apply 
Theorem~\ref{thm:noladder} to see that each part of $B$ is a circular ordering of short~$\Theta$'s. Moreover, if the longest cycle in a part has length at most~$4$, then that part is also a short~$\Theta$.
Since the degree of a passing point~$p$ is larger than~$2$ in any 
part~$R$ that contains~$p$, it follows that~$p$ has to be an endvertex of a short~$\Theta$ in~$R$.

Suppose there are $\ell$ edge-disjoint cycles $C_1,\ldots, C_\ell$ of length at least $5$ such that~$C_i$ is contained entirely inside a part~$R_i$. As the length of a longest cycle inside a short~$\Theta$ is at most~$4$, for $i\in[\ell]$, the cycle $C_i$ has to pass through all endvertices of short~$\Theta$'s in~$R_i$. In particular, the cycle $C_i$ passes through the passing points of $R_i$. The parts that contain~$v$ are short~$\Theta$'s, which means that $v\notin V(R_i)$ for each $i\in[\ell]$. In each part that is a circular ordering of short~$\Theta$'s, we can go from one passing point to the other through two different sets of short~$\Theta$'s. Removing the edges of any cycle inside such a part can only reduce the number of edge-disjoint paths through each set between the passing points by at most two (it is reduced by two only when the cycle passes through a diamond `sideways' or when a cycle is completely contained within a short $\Theta$). Hence, whenever the edges of a cycle inside a part are removed, the maximum number of edge-disjoint paths between the passing points of $R_i$ decreases by at most~$4$. 

Since in each part of $B$ there are $5\ell$ edge-disjoint paths between its passing points, there are at least $\ell$ edge-disjoint paths between the passing points of each part of $B$ that are edge-disjoint from $C_1,\ldots,C_\ell$.
For $i\in[\ell]$, merge $C_i$ with one path between the passing points of each part of $B$ that is distinct from $R_i$. The result is the cycle $C_i$ together with a $V(C_i)$-path that passes through $v$ and as $v\notin V(R_i)$, has length at least~$4$.
(We use here that $v$ is not adjacent to $w_1$ and $w_2$ in $G'$.)
We already remarked in the proof of Lemma \ref{thm:noladder} that a cycle of length at least~$5$ together with a path of length at least~$3$ is a ladder. As there are $\ell$ edge-disjoint cycles $C_1,\ldots,C_\ell$ and as there are $\ell$ edge-disjoint paths between the passing points of each part that are disjoint from the cycles $C_1,\ldots,C_\ell$, we obtain $\ell$ edge-disjoint ladders.

So we may assume that there are fewer than $\ell$ edge-disjoint cyles of length at least~$5$ in the parts of $B$.
Recall that a circular ordering of short $\Theta$'s that does not contain a cycle of length at least~$5$ is also a short $\Theta$.
Let~$n$ be the maximum number of edge-disjoint cycles of length at least~$5$ in a part~$R$ and suppose that $n\geq 1$. Let $r$ be the minimum number such that one of the short $\Theta$'s of $R$ is a short $\Theta_r$.
By construction of a circular ordering of short~$\Theta$'s, it is not that hard to see that $\frac{r}{2}\leq n\leq r$. This then implies that there is a set of at most $r\leq 2n$ edges inside that $\Theta_r$ that intersects all cycles of length at least~$5$ in $R$.

As the number of cycles of length at least~$5$ in the parts of $B$ is less than $\ell$, we
find a set $E$ of at most $2\ell$ edges such that for each part $R$ of $B$, the blocks of $R-E$ are short~$\Theta$'s.
Because of the cyclic structure seen in Figure~\ref{parts},
the block of $B-E$ that contains $v$ is a circular ordering of short $\Theta$'s now, which by Lemma \ref{thm:noladder}, implies that there are no ladders in $B-E$ anymore. This finishes the proof of Claim~5 and thus, the proof of the theorem.
\qed
\end{proof}

\section{Discussion} \label{sec:discussion}

Is Theorem~\ref{no14rungs} optimal? Proposition~\ref{tools:prop:13rungLadders} shows that it is impossible to improve upon our bound of $14$~rungs when using a condensed wall. 

\begin{proposition} \label{tools:prop:13rungLadders}
For every $n \in \N$, every condensed wall $W$ of size $r \geq 5 \cdot n$ contains $n$ edge-disjoint subdivisions of ladders $L_1, L_2, \ldots , L_n$ which all have exactly 13~rungs.
\end{proposition}

\begin{proof}
We split our wall into $n$ subgraphs $C_i, i \in [n]$ each containing exactly 5~layers of $W$ plus $a$ and $b$. If $r > 5 \cdot n$, we can discard any additional layers by not placing any part of a ladder in them. Let $C_i$ be the subgraph of $W$ induced by $W_{5 \cdot i - 4}, W_{5 \cdot i - 3}, \ldots , W_{5 \cdot i}, a$ and $b$ for all $i \in \N$.
The only vertices shared by all $C_i$ are $a$ and $b$, otherwise $C_i$ and $C_{i+1}$ only overlap at a bottleneck vertex.
In particular, $C_i$ and $C_j$ are edge-disjoint for every $i \neq j$.

Each ladder $L_i, i \in [n]$ can be embedded in $C_i$ 
as in Figure~\ref{fig:LadderInWall}.


\begin{figure}[bht] 
\centering
\begin{tikzpicture}[scale=1]
\tikzstyle{tinyvx}=[thick,circle,inner sep=0.cm, minimum size=1.5mm, fill=white, draw=black]

\def\vstep{1}
\def\hstep{0.5}
\def\hwidth{9}
\def\hheight{4}

\def\totalheight{\hheight*\vstep}
\def\totalwidth{\hwidth*\hstep}
\pgfmathtruncatemacro{\minustwo}{\hwidth-2}
\pgfmathtruncatemacro{\minusone}{\hwidth-1}

\foreach \j in {0,...,\hheight} {
\draw[ledge] (0,\j*\vstep) -- (\hwidth*\hstep,\j*\vstep);
\foreach \i in {0,...,\hwidth} {
\node[tinyvx] (v\i\j) at (\i*\hstep,\j*\vstep){};
}
}

\foreach \j in {1,...,\hheight}{
\node[hvertex] (z\j) at (0.5*\hwidth*\hstep,\j*\vstep-0.5*\vstep) {};
}
\pgfmathtruncatemacro{\plusvone}{\hheight+1}

\node[hvertex,fill=hellgrau,label=above:$c$] (z\plusvone) at (0.5*\totalwidth,\totalheight+0.5*\vstep) {};
\node[hvertex,fill=hellgrau,label=below:$d$] (z0) at (0.5*\totalwidth,-0.5*\vstep) {};

\foreach \j in {1,...,\plusvone}{
\pgfmathtruncatemacro{\subone}{\j-1}
\draw[line width=1.3pt,double distance=1.2pt,draw=white,double=black] (z\j) to (z\subone);
\foreach \i in {0,2,...,\hwidth}{
\draw[ledge] (z\j) to (v\i\subone);
}
}

\foreach \j in {0,...,\hheight}{
\foreach \i in {1,3,...,\hwidth}{
\draw[ledge] (z\j) to (v\i\j);
}
}

\pgfmathtruncatemacro{\minusvone}{\hheight-1}
\node[hvertex,fill=hellgrau,label=left:$a$] (a) at (-\hstep,0.5*\totalheight) {};
\foreach \j in {0,...,\hheight} {
\draw[ledge] (a) to (v0\j);
}

\node[hvertex,fill=hellgrau,label=right:$b$] (b) at (\totalwidth+\hstep,0.5*\totalheight) {};
\foreach \j in {0,...,\hheight} {
\draw[ledge] (v\hwidth\j) to (b);
}

	\draw[rededge] (v64) -- (v54) node [midway, yshift=-0.25cm] {$1$}; 
	\draw[rededge] (v54) to (v44) ;
	\draw[rededge] (v64) to (z5)  ;
	\draw[rededge] (v44) -- (z5)  node [midway] {$2$}; 
	\draw[rededge] (v44) to (v34) ;
	\draw[rededge] (v24) to (z5)  ;
	\draw[rededge] (v34) -- (v24) node [midway] {$3$}; 
	\draw[rededge] (v34) to (z4)  ;
	\draw[rededge] (v24) to (v14) ;	
	\draw[rededge] (v14) -- (z4)  node [midway, yshift=-0.15cm, label=left:{$4$}] {}; 
	\draw[rededge] (v14) to (v04) ;	
	\draw[rededge] (a)   to (v04) ;
	\draw[rededge] (z4)  to (v03) ;
	\draw[rededge] (a)   -- (v03) node [midway, label=right:{$5$}] {}; 
	\draw[rededge] (a)   to (v02) ;
	\draw[rededge] (v03) to (v13) ;
	\draw[rededge] (v13) to (z3)  ;
	\draw[rededge] (v02) -- (z3)  node [midway, yshift=0.15cm, label=left:{$6$}] {}; 
	\draw[rededge] (v02) to (v12) ;
	\draw[rededge] (z3)  to (v22) ;
	\draw[rededge] (v12) -- (v22) node [midway] {$7$}; 
	\draw[rededge] (v12) to (z2)  ;
	\draw[rededge] (v22) to (v32) ;
	\draw[rededge] (v32) -- (z2)  node [midway] {$8$}; 
	\draw[rededge] (z2)  to (v81) ;
	\draw[rededge] (v81) to (v91) ;
	
	\draw[rededge] (v32) to (v42) ;
	\draw[rededge] (v42) to (v52) ;
	\draw[rededge] (v52) to (v62) ;
	\draw[rededge] (v62) to (v72) ;
	\draw[rededge] (v72) to (v82) ;
	\draw[rededge] (v82) to (v92) ;
	\draw[rededge] (v92) to (b)   ;
	\draw[rededge] (b)   -- (v91) node [midway, label=left:{$9$}] {}; 
	\draw[rededge] (v91) to (z1)  ;
	\draw[rededge] (b)   to (v90) ;
	\draw[rededge] (v80) to (v90) ;
	\draw[rededge] (v80) -- (z1)  node [midway, yshift=0.15cm, label=right:{$10$}] {}; 
	\draw[rededge] (v80) to (v70) ;
	\draw[rededge] (v60) to (z1)  ;
	\draw[rededge] (v70) -- (v60) node [midway] {$11$}; 
	\draw[rededge] (v70) to (z0)  ;
	\draw[rededge] (v60) to (v50) ;
	\draw[rededge] (v50) -- (z0)  node [midway] {$12$}; 
	\draw[rededge] (v30) to (z0)  ;
	\draw[rededge] (v50) to (v40) ;
	\draw[rededge] (v40) -- (v30) node [midway, yshift=0.25cm] {$13$}; 

\end{tikzpicture}
\caption{A ladder (thick edges) of size~13 in a condensed wall of size 5. Edges belonging to rungs are labeled.}
\label{fig:LadderInWall}
\end{figure}
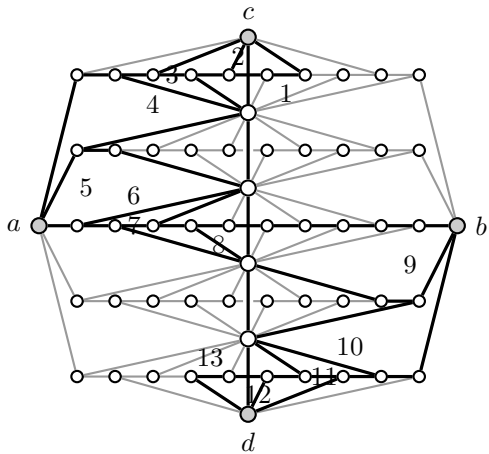

This yields a ladder $L_i$ with 13~rungs for every $i \in [n]$ and those ladders do not share a single edge. Thus we have proven that $W$ contains $n$ edge-disjoint ladders.
\qed
\end{proof}

In Conjecture~\ref{intro:conjectureCondWall}, we have posed the question whether the condensed wall plays a key role for characterizing the edge-Erd\H{o}s-P\'{o}sa property. 
We have completed an alternative proof for Theorem~\ref{no14rungs} (which can be found in \cite{steck23dissertation})
which does not use a standard condensed wall anymore, but omits all jump-edges $z_{i-1} z_{i}$. 
This is a hint that the condensed wall might not be as crucial as it seemed to be.

Curiously, the proof in \cite{steck23dissertation} is also the only known case for which the edge-Erd\H{o}s-P\'{o}sa property does not fail for $k=2$, but instead for $k=3$.


With respect to ladders, we can also take a look at the other side of the problem: Can we prove that ladders with more than three rungs have the edge-Erd\H{o}s-P\'{o}sa property?
We do not know, but at least the proof in this paper will not carry over to a larger ladder.
This is because the way we construct the tree in the proof does not help for larger ladders. Additonally, it looks like it is much harder to characterize the graphs without longer ladders. However, it might be possible to generalize the approach to other subdivisions of $\Theta_3$.

We also want to quickly discuss the hitting set size for ladders with three rungs.
Let $\mathcal{H}$ be the set of ladders with three rungs. The hitting set bound for these in $\mathcal{G}_\mathcal{H}^*$ is at least $k-1$ (take for example $k-1$ ladders and choose one vertex from each ladder and identify them with each other) and as we have mentioned in the proof of Theorem \ref{3RungLadderHasEdgeEPProperty} the optimal vertex hitting set bound for the vertex-\EPP~of ladders is $\Theta(k\log k)$ \cite{BHJR19}. So if we are using Lemma \ref{1vhs} the hitting set bound will be in $\Omega(k^2\log k)$. We could quite easily obtain this bound if we could obtain a linear bound in $k$ for Theorem \ref{thm:AmTree}, that is for $(A, m)$-trees, as then the hitting set bound in $\mathcal{G}_\mathcal{H}^*$ would also be linear instead of quadratic. And actually we do believe that this is true. On the other hand the bound for the ladders has to be in $\Omega(k\log k)$ as in the vertex version \cite{RST16}, so there is not that much room to improve. 


\section*{Data Availability}

Data sharing not applicable to this article as no datasets were generated or analysed during the current study.

%
\section*{Conflict of interest}

The authors have no relevant financial or non-financial interests to disclose.

\bibliographystyle{spmpsci}      
\bibliography{literature}{}

\end{document}